\documentclass[a4paper,11pt]{amsart}
\usepackage{color}
\usepackage{amsmath,amscd,amsthm,upref}
\usepackage[psamsfonts]{amssymb}
\usepackage{bm}
\usepackage[all]{xy}
\usepackage{enumitem}
\usepackage{comment}
\usepackage[square,comma]{natbib}

\numberwithin{equation}{section}

\newtheorem{theorem}{Theorem}[section]
\newtheorem{corollary}[theorem]{Corollary}
\newtheorem{lemma}[theorem]{Lemma}
\newtheorem{proposition}[theorem]{Proposition}
\theoremstyle{definition}
\newtheorem{definition}[theorem]{Definition}
\newtheorem{example}[theorem]{Example}
\theoremstyle{remark}
\newtheorem*{remark}{Remark}

\newcommand{\lsup}[2]{\vphantom{#2}^{#1}{#2}}
\newcommand{\nst}[1]{\lsup{*}{#1}}
\newcommand{\asy}[1]{\lsup{\mathit{a}}{#1}}
\newcommand{\bas}[1]{\lsup{\bullet}{#1}}
\newcommand{\ras}[1]{\lsup{\rho}{#1}}
\newcommand{\col}[1]{\lsup{\mathit{s}}{#1}}
\DeclareMathOperator{\bM}{\bas{\hspace{-.5pt}\mathbf M}}

\DeclareMathOperator{\nM}{\nst{\hspace{-.5pt}\mathbf M}}
\DeclareMathOperator{\nN}{\nst{\hspace{-.5pt}\mathbf N}}
\DeclareMathOperator{\rL}{\ras{\hspace{-.5pt}\mathbf L}}
\DeclareMathOperator{\rM}{\ras{\mathbf M}}
\DeclareMathOperator{\rN}{\ras{\mathbf N}}
\DeclareMathOperator{\aL}{\asy{\hspace{-.5pt}\mathbf L}}
\DeclareMathOperator{\aM}{\asy{\mathbf M}}
\DeclareMathOperator{\aN}{\asy{\mathbf N}}
\DeclareMathOperator{\cL}{\col{\hspace{-.5pt}\mathbf L}}
\DeclareMathOperator{\cM}{\col{\mathbf M}}
\DeclareMathOperator{\cN}{\col{\mathbf N}}
\DeclareMathOperator{\sM}{\col{\!\mathcal M}}
\DeclareMathOperator{\sN}{\col{\!\mathcal N}}
\DeclareMathOperator{\Gs}{\mathcal{G}^{\mathit s}}
\DeclareMathOperator{\I}{\mathcal{I}}
\DeclareMathOperator{\Coeq}{Coeq}
\DeclareMathOperator{\Eq}{Eq}
\DeclareMathOperator*{\colim}{colim}
\DeclareMathOperator{\Diff}{\bf Diff}
\DeclareMathOperator{\bDiff}{\bas{\Diff}}

\DeclareMathOperator{\rDiff}{\ras{\Diff}}
\DeclareMathOperator{\wDiff}{\widehat{\Diff}}
\DeclareMathOperator{\ev}{ev}
\DeclareMathOperator{\Image}{Im}
\DeclareMathOperator{\tr}{Tr}
\DeclareMathOperator{\dg}{Dg}

\DeclareMathOperator{\Path}{\widehat{\mathcal{P}}}

\DeclareMathOperator{\pr}{pr}

\DeclareMathOperator{\Set}{\bf Set}
\DeclareMathOperator{\supp}{supp}
\DeclareMathOperator{\Top}{\bf Top}
\DeclareMathOperator{\EucOp}{\bf EucOp}
\DeclareMathOperator{\bEucOp}{\bas{\EucOp}}
\DeclareMathOperator{\nEucOp}{\nst{\EucOp}}
\DeclareMathOperator{\rEucOp}{\ras{\EucOp}}
\newcommand{\norm}[1]{\|{#1}\|}
\newcommand{\N}{\mathbb{N}}
\newcommand{\Q}{\mathbb{Q}}
\newcommand{\C}{\mathbb{C}}
\newcommand{\aC}{\asy{\mathbb{C}}}
\newcommand{\bC}{\bas{\mathbb{C}}}
\newcommand{\nC}{\nst{\mathbb{C}}}
\newcommand{\rC}{\ras{\mathbb{C}}}
\newcommand{\bD}{\bas{\!D}}
\newcommand{\rD}{\ras{\!D}}
\newcommand{\R}{\mathbb{R}}
\newcommand{\aR}{\asy{\mathbb{R}}}
\newcommand{\bR}{\bas{\mathbb{R}}}
\newcommand{\nR}{\nst{\mathbb{R}}}
\newcommand{\rR}{\ras{\mathbb{R}}}
\newcommand{\F}{\mathbb{F}}
\newcommand{\aF}{\asy{\mathbb{F}}}
\newcommand{\bF}{\bas{\mathbb{F}}}
\newcommand{\nF}{\nst{\mathbb{F}}}
\newcommand{\rF}{\ras{\mathbb{F}}}
\newcommand{\D}{\mathcal{D}}
\newcommand{\U}{\mathcal{U}}
\newcommand{\bI}{\partial I}
\newcommand{\dint}{\displaystyle\int}

\newcommand{\tcoprod}{\textstyle\coprod}
\newcommand{\tprod}{\textstyle\prod}
\newcommand{\tsum}{\textstyle\sum}
\newcommand{\aev}{\text{\rm a.e.}}
\newcommand{\abs}[1]{\left\lvert{#1}\right\rvert}
\newcommand{\bracket}[2]{\langle{#1}\,|\,{#2}\rangle}
\newcommand{\Cinf}{\mathit{C}^{\infty}}
\newcommand{\bCinf}[1][\infty]{\bas{\hspace{-1pt}\mathit{C}}^{#1}}
\newcommand{\nCinf}[1][\infty]{\nst{\hspace{-2pt}\mathit{C}}^{#1}}
\newcommand{\rCinf}[1][\infty]{\ras{\hspace{-1pt}\mathit{C}}^{#1}}
\newcommand{\tCinf}{\mathcal{C}^{\infty}}
\newcommand{\wCinf}{\widehat{\mathit{C}}^{\infty}}
\title{Nonstandard diffeology and generalized functions} %
\author[K. Shimakawa]{Kazuhisa Shimakawa}%
\address{ Okayama University \\ Okayama 700-8530 \\ Japan}%
\email{kazuhisa.shimakawa@gmail.com}%
\date{\today}%
\subjclass[2010]{46T30, 54C35, 55U40, 58D15}%
\keywords{Schwartz distribution, generalized function,
  diffeological space}%
\begin{document}
\maketitle
%%%%% Abstract
\begin{abstract}
  We introduce a nonstandard extension of the category of
  diffeological spaces, and demonstrate its application to
  the study of generalized functions.
  Just as diffeological spaces are defined as concrete
  sheaves on the site of Euclidean open sets, our
  nonstandard diffeological spaces are defined as concrete
  sheaves on the site of open subsets of nonstandard
  Euclidean spaces, i.e.\ finite dimensional vector spaces
  over (the quasi-asymptotic variant of) Robinson's
  hyperreal numbers.
  It is shown that nonstandard diffeological spaces form a
  category which is enriched over the category of
  diffeological spaces, is closed under small limits and
  colimits, and is cartesian closed.
  %
  % Furthermore, it is shown that the space of nonstandard
  % smooth functions on (the extension of) a Euclidean open
  % set is a smooth differential algebra that admits an
  % embedding of the differential vector space of Schwartz
  % distributions.
  % %
  % Since our algebra of generalized functions comes as a
  % hom-object in a category, it enables not only the
  % multiplication of distributions but also the composition
  % of them.
  % %
  % To illustrate the usefulness of this feature, we show
  % that the homotopy extension property can be established
  % for smooth relative cell complexes by exploiting extended
  % maps.
  Furthermore, it is shown that the differential vector
  space of Schwartz distributions on an open subset $U$ of a
  Euclidean space can be embedded into the space of
  nonstandard smooth functions defined on the natural
  extension of $U$.
  Since our algebra of generalized functions comes as a
  hom-object in a category, it enables not only the
  multiplication of distributions but also the composition
  of them.
  To illustrate the usefulness of this feature, we show that
  the homotopy extension property can be established for
  smooth relative cell complexes by exploiting extended
  maps.
\end{abstract}

%%%%%%%%%%%%%%%%%%%%%%%%%%%%%%%%%%%%%%%%%%%%%%%%%%%%%%%%%%%%
% Section 1
%%%%%%%%%%%%%%%%%%%%%%%%%%%%%%%%%%%%%%%%%%%%%%%%%%%%%%%%%%%%
\section{Introduction}
Schwartz distributions make it possible to differentiate
functions whose derivatives do not exist in the classical
sense, e.g.\ locally integrable functions.
They are widely used in the theory of partial differential
equations, where it may be easier to establish the existence
of weak solutions than classical ones, or appropriate
classical solutions may not exist.
%
% They are also important in physics and engineering where
% many problems naturally lead to differential equations whose
% solutions or initial conditions are distributions.
% (multiplication)
However, there is a serious drawback that distributions
cannot be multiplied nor composed except for very special
cases.  Thus, the formulas such as
\[
  (e^{\delta(x)})' = e^{\delta(x)}\delta'(x) \quad
  \text{($\delta$: Dirac's delta)}
\]
do not make sense within the framework of Schwartz
distributions.

A viable way to resolve the problem of multiplication will
be to construct a differential algebra which includes the
space of distributions $\D'(U)$ as a linear subspace.  At
first glance, this does not look feasible considering the
Schwartz impossibility result stating that no differential
algebra containing the space of distributions preserves the
product of continuous functions.
Nevertheless, if we only wish to preserve the product of
smooth functions then the construction of such an algebra is
possible, the best known example of which is the algebra of
generalized functions introduced by J.~F.~Colombeau.
However, Colombeau's algebra does not solve the second
problem, for it is not closed under composition.

Unlike the case of multiplicativity, composability of
distributions has not attracted much attention.  But
recently, \citep{Giordano} created a solution to the problem
of composability by introducing a category which has as its
morphisms smooth set-theoretic maps on (multidimensional)
points of a ring of scalars having infinitesimals and
infinities, and includes Schwartz distributions.
The aim of this article is to present an yet another
approach to constructing a category of generalized maps
containing Schwartz distributions.
Giordano's approach is based on Colombeau's algebra of
generalized functions and employs as the scalar a partially
ordered ring.  On the other hand, we use Robinson's field of
nonstandard numbers as the scalar field.  While this brings
us simplicity in the construction of the theory, it
complicates the relation between Colombeau's algebra and
ours.  Still, we can construct a chain of homomorphisms
relating Colombeau's algebra to our algebra of
nonstandard-valued smooth functions.

The paper is organized as follows.  In Section 2 we first
recall the basic properties of the category $\Diff$ of
diffeological spaces and its relation with the category of
topological spaces given by the adjunction
$L \dashv R \colon \Diff \rightleftarrows \Top$.
Then, we introduce two variants of the nonstandard number
field: one is the field of hyperreals $\nR$ (and its
complexification $\nC$), and the other is the field of
{quasi-asymptotic} real numbers $\rR$ (and its
complexification $\rC$).
Like the field of asymptotic numbers, $\rF$
($\F = \R,\, \C$) can be written as a quotient
$\rM(\nF)/\rN(\nF)$ of subalgebras
$\rN(\nF) \subset \rM(\nF) \subset \nF$.  But $\rF$ is
larger than the field of asymptotic numbers in the sense
that the latter is a subquotient of the former.
It turns out that for both $\bullet = *$ and $\rho$, $\bR$
is a non-Archimedean real closed field and $\bC$ is an
algebraically closed field of the form
$\bC = \bR + \sqrt{-1}\,\bR$.
% Moreover, via the ultrapower construction we can show that
% $\bF$ is a smooth field with respect to the diffeology
% derived from the standard diffeology of $\F$.  As a
% consequence, we see that $\bF$ is Hausdorff with respect
% to the $D$-topology and $\F$ is discrete in $\bF$.

In Section 3 we develop differential calculus on the
nonstandard Euclidean space $\bR^k$.  % ($k \geq 0$)
We introduce on $\bR^k$ a Hausdorff topology generated by
infinitesimal open neighborhoods.
Then the ultrapower construction enables us to define
partial derivatives of a function $f \colon U \to \bF$
defined on an open subset $U$ of $\bR^k$ and, consequently,
a differential algebra of infinitely differentiable
functions over $\bF$, which we denote $\bCinf(U,\bF)$
($\bullet = *,\, \rho$).
If $V$ is open in $\R^k$ then so is its natural extension
$\bas{V} \subset \bR^k$, and there is a natural inclusion of
differential algebras $\Cinf(V,\F) \to \bCinf(\bas{V},\bF)$.
Most of the basic properties of usual differentiable
functions are also shared by nonstandard ones; in
particular, the Intermediate Value Theorem and the Mean
Value Theorem hold for nonstandard differentiable functions.
Once the notion of infinite differentiability of
$\bF$-valued functions is established, it is easy to extend
it to that of maps between open subsets of nonstandard
Euclidean spaces, and we have a set of infinitely
differentiable maps $\bCinf(U,V)$ between open subsets
$U \subset \bR^k$ and $V \subset \bR^l$.  As in the
classical case, again, there is a well defined composition
\[
  \bCinf(V,W) \times \bCinf(U,V) \to \bCinf(U,W) \quad
  (U,\, V,\, W \in \bEucOp)
\]
and we obtain a concrete site $\bEucOp$ equipped with the
coverage consisting of open covers.  Moreover, there is a
faithful functor $\EucOp \to \bEucOp$ which takes an open
subset $U \subset \R^k$ to the corresponding internal open
subset $\bas{U} \subset \bR^k$ and $f \in \Cinf(U,V)$ to its
natural extension $\bas{\!f} \in \bCinf(\bas{U},\bas{V})$.

Based on the results of the preceding sections, we define
the category $\bDiff$ ($\bullet = *,\, \rho$) of nonstandard
diffeological spaces as the category of concrete sheaves on
$\bEucOp$.  We see in Section~4 that $\bDiff$ is closed
under small limits and colimits, is enriched over itself and
is cartesian closed with hom-objects as exponentials.
%
% Moreover, $\bDiff$ is related to $\Diff$ and $\Top$ via
% adjunctions
% $\tr \dashv \dg \colon \Diff \rightleftarrows \bDiff$ and
% $\bas{\!L} \dashv \bas{\!R} \colon \bDiff \rightleftarrows
% \Top$.  In particular, together with the help of the natural
% transformation
% $L \circ \dg \to \bas{\!L} \colon \bDiff \to \Top$ we can
% show that the natural map
% $\Cinf(X,Y) \to \bCinf(\bas{\!X},\bas{Y})$ is smooth, hence
% continuous with respect to the topologies given by $L$ on
% the source and $\bas{\!L}$ on the target.
Moreover, there are left adjoint functors
$\tr \colon \Diff \to \bDiff$ extending the inclusion
$\EucOp \to \bEucOp$ and $\bas{\!L} \colon \bDiff \to \Top$
assigning the underlying topology.  By deploying these
adjunctions, we can show that the map
$\Cinf(X,Y) \to \bCinf(\bas{\!X},\bas{Y})$ which takes
$f \colon X \to Y$ to
$\tr(f) = \bas{\!f} \colon \bas{\!X} \to \bas{Y}$ is smooth,
and hence continuous with respect to the topologies given by
$L$ on the source and $\bas{\!L}$ on the target.

In Section~5 we construct a continuous linear embedding
\[
  \I_U \colon \D'(U) \to \rCinf(\ras{U},\rF) \quad (U \in
  \EucOp)
\]
which extends the inclusion
$\Cinf(U,\F) \to \rCinf(\ras{U},\rF)$.  The construction of
$\I_U$ is analogous to the construction of the embedding
$\D'(U) \to \Gs(U)$ into Colombeau's special algebra
described in \citep{GKOS}: We first construct an embedding
of the space of compactly supported distributions and expand
it over $\D'(U)$ by employing sheaf-theoretic argument.
The mutual relation between $\rCinf(\ras{U},\rF)$ and
$\Gs(U)$ can be clarified through their relations to the
algebra $\ras{\mathbb{E}(U)}$ of asymptotic functions
introduced by \citep{Oberguggenberger-Todorov}.
We show in Section~6 that $\ras{\mathbb{E}}(U)$ is a
subquotient of $\rCinf(\ras{U},\rF)$ and there is a
homomorphism of differential algebras
$\Gs(U) \to \ras{\mathbb{E}}(U)$ that respects the
embeddings of $\D'(U)$.
Note however that, $\rCinf$ machinery has an advantage over
$\Gs$ and $\ras{\mathbb{E}}$ ones in that it enables the
composition of (vector-valued) distributions as the
composition of morphisms in a category.
Neither $\Gs$ nor $\ras{\mathbb{E}}$ can be extended to a
composable system.
% In order to attain composability,we need
% quasi-asymptoticity instead of asymptoticity.

Finally, in Section~7 we present several examples exhibiting
that the use of extended morphism provides a flexible
environment to study homotopy theory of diffeological
spaces.
More specifically, we consider those set maps $X \to Y$ that
extends to a morphism $\ras{\!X} \to \ras{Y}$ in $\rDiff$.
Such maps are called {quasi-asymptotic maps} from $X$ to
$Y$.  Smooth maps are evidently quasi-asymptotic and, more
generally, so are piecewise smooth maps.  By extending
smooth maps to quasi-asymptotic ones we can establish e.g.\
strict concatenation of paths in a space and homotopy
extension property for smooth relative cell complexes.

% \subsection*{Acknowledgment}
The author wishes to thank Paolo Giordano for his interest
and valuable suggestions during the preparation of the
article.  He also thanks Dan Christensen and Katsuhiko
Kuribayashi for helpful comments.

%%%%%%%%%%%%%%%%%%%%%%%%%%%%%%%%%%%%%%%%%%%%%%%%%%%%%%%%%%%%
% Section 2
%%%%%%%%%%%%%%%%%%%%%%%%%%%%%%%%%%%%%%%%%%%%%%%%%%%%%%%%%%%%
\section{Smooth fields of nonstandard numbers}
\subsection{The category of diffeological spaces}
We briefly recall the basic definitions and properties of
diffeological spaces.  For details, see \citep{Zemmour}.

Denote by $\EucOp$ the site consisting of open sets in the
Euclidean spaces and smooth (i.e.\ infinitely
differentiable) maps between them endowed with the coverage
consisting of open covers.

% \medskip
\begin{definition}
  A {diffeological space} is a concrete sheaf on $\EucOp$.
  A {smooth map} between diffeological spaces is a morphism
  between the corresponding sheaves.  The category
  consisting of diffeological spaces and smooth maps is
  denoted by $\Diff$.
\end{definition}
% \medskip

% Observe that any diffeological space $X$ is a subsheaf of
% the sheaf $U \mapsto \hom_{\Set}(U,\abs{X})$ and contains as
% its subsheaf the locally constant sheaf
% $U \mapsto \{ \sigma \in \hom_{\Set}(U,\abs{X}) \mid
% \text{$\sigma$ is locally constant} \}$.

Given a diffeological space $X$, let us write
$\abs{X} = X(\R^0)$ and call it the {underlying set} of $X$.
Then each section $\sigma \in X(U)$ determines and is
determined by a set map $U \to \abs{X}$ which takes
$u \in U$ to the image of $\sigma$ under the map
$X(U) \to X(\R^0) = \abs{X}$ induced by $\R^0 \to U$,
$0 \mapsto u$.
Thus we arrive at an alternative (in fact, more familiar)
definition of a diffeological space as a pair
$(X,\mathcal{D})$ consisting of a set $X$ and the set of its
plots
$\mathcal{D} \subset \coprod_{U \in \EucOp}
\hom_{\Set}(U,X)$ subject to the conditions:
\begin{enumerate}[topsep=4pt]
\item % (Covering)
  Every constant map $\mathbb{R}^n \to X$ belongs to
  $\mathcal{D}$.
\item % (Locality)
  A map $\sigma \colon U \to X$ is in $\mathcal{D}$ if and
  only if it is locally so.
\item % (Smooth Compatibility)
  If $\sigma \colon U \to X$ belongs to $\mathcal{D}$ then
  so does $\sigma \circ \phi \colon V \to X$ for any smooth
  map $\phi \colon V \to U$.
\end{enumerate}
In this context, a smooth map from $(X,\mathcal{D})$ to
$(Y,\mathcal{D}')$ can be defined as a set map
$f \colon X \to Y$ such that
$f \circ \sigma \in \mathcal{D}'$ holds for every
$\sigma \in \mathcal{D}$.
From now on, we identify each section $\sigma \in X(U)$ with
the corresponding plot which we simply denote
$\sigma \colon U \to X$ instead of
$\sigma \colon U \to \abs{X}$ unless distinction is
necessary.
Observe that every plot $\sigma \colon U \to X$ can be
written as a composition
\begin{equation}
  \label{eq:underlying_set_as_quotient}
  U = U \times \{\sigma\}
  % \xrightarrow{\subset} U \times X(U)
  \xrightarrow{\subset} \underset{U \in
    \EucOp}{\tcoprod} U \times X(U) \xrightarrow{\pr}
  \underset{U \in \EucOp}{\tcoprod} U \times X(U)/\!\sim \ \
  \cong \ \abs{X}
\end{equation}
where $\pr$ is the projection with respect to the
equivalence relation generated by
$(u,\tau \circ \phi) \sim (\phi(u),\tau)$ for $u \in U$,
$\tau \in X(V)$ and $\phi \in \Cinf(U,V)$, and the last
bijection identifies the class of
$(u,\sigma) \in U \times X(U)$ with $\sigma(u) \in \abs{X}$.

The following constructions play crucial role in later
discussions.

\paragraph{Subspace}
A smooth injection $f \colon X \to Y$ is called an induction
if $X$ is the pullback of $Y$ by $f$, i.e.\
\[
  X(U) = f^*Y(U) := \{ \sigma \in \hom_{\Set}(U,X) \mid f
  \circ \sigma \in Y(U) \} \quad (U \in \EucOp).
\]
% A subspace $A$ of $X$ is a subset of $\abs{X}$ equipped with
% the diffeology such that the inclusion $A \to X$ is an
% induction.
A (diffeological) subspace of $X \in \Diff$ is an arbitrary
subset $A$ of the underlying set of $X$ equipped with the
diffeology such that the inclusion $A \to X$ is an
induction.

\paragraph{Quotient space}
A smooth surjection $f \colon X \to Y$ is called a
subduction if for any $U \in \EucOp$, $Y(U)$ is isomorphic
to the pushforward $f_*X(U)$ of $X(U)$ by $f$, that is,
\[
  f_*X(U) = \{ \sigma \in \hom_{\Set}(U,Y) \mid \forall u
  \in U,\ \exists \tau \in X(V)\ (u \in V \subset U),\
  \sigma|V = f \circ \tau \}.
\]
If $\mathcal{R}$ is an equivalence relation on the
underlying set of $X$ then the quotient space
$X/\mathcal{R}$ is the set of equivalence classes
$\abs{X}/\mathcal{R}$ equipped with the diffeology such that
the projection $X \to X/\mathcal{R}$ is a subduction.

\paragraph{Product}
Given a family of diffeological spaces $\{X_j\}_{j \in J}$
their product $\prod_{j \in J} X_j$ is the sheaf
\[
  U \mapsto \tprod_{j \in J} X_j(U) \quad (U \in \EucOp).
\]
For each $k \in J$ the $k$-th projection
$\prod_{j \in J} X_j \to X_k$ is the sheaf morphism
$\prod_{j \in J} X_j(U) \to X_k(U)$.

\paragraph{Coproduct}
The coproduct $\coprod_{j \in J}X_j$ is defined as the
sheafification of the presheaf
\[
  U \mapsto \tcoprod_{j \in J} X_j(U) \quad (U \in
  \EucOp).
\]
More explicitly, a set map
$\sigma \colon U \to \coprod_{j \in J} X_j$ is a plot of
$\coprod_{j \in J}X_j$ if for every $x \in U$ there exist a
neighborhood $V \subset U$ and a plot $\tau \in X_j(V)$ for
some $j$ such that $\tau = \sigma|V$ holds.
% The inclusion $X_j \to \coprod_{j \in J}X_j$ is given by
% the maps $\sigma \mapsto i_j \circ \sigma$ for
% $\sigma \in X_j(U)$.

\paragraph{Hom-object}
Given $X,\, Y \in \Diff$, $\Cinf(X,Y)$ denotes the set of
smooth maps $X \to Y$ equipped with the coarsest diffeology
such that the evaluation map
$\ev \colon \Cinf(X,Y) \times X \to Y$ is smooth.
Explicitly, we have
\[
  \Cinf(X,Y)(U) := \{ \sigma \colon U \to \hom_{\Diff}(X,Y)
  \mid \forall \tau \in X(U),\ (u \mapsto \sigma(u)(\tau(u))
  \in Y(U) \}.
\]

\paragraph{$D$-Topology}
Given a diffeological space $X$, its {$D$-topology} is the
finest topology on $\abs{X}$ such that every plot of $X$ is
continuous; or equivalently, the quotient topology with
respect to the projection
$\coprod U \times X(U) \to \abs{X}$ (see
(\ref{eq:underlying_set_as_quotient})\,), where each
$U \in \EucOp$ is endowed with standard metric topology.
Clearly, any smooth map between diffeological spaces induces
a continuous map between underlying sets equipped with
$D$-topology.  Hence there is a functor
$L \colon \Diff \to \Top$ which assigns to any diffeological
space its underlying set equipped with $D$-topology.  Its
right adjoint $R \colon \Top \to \Diff$ is given by
\[
  RY(U) = \hom_{\Top}(U,Y) \quad (Y \in \Top,\ U \in \EucOp)
\]
that is, $RY$ has the same underlying set as $Y$ and all the
continuous maps $U \to Y$ as its plots.
The unit of the adjunction $\eta \colon X \to RLX$ is given
by the natural inclusion $X(U) \subset RLX(U)$ which
identifies each plot $U \to X$ of $X$ with the corresponding
continuous map $U \to LX$ regarded as a plot of $RLY$.
On the other hand, the counit $\varepsilon \colon LRY \to Y$
is given by the identity of the underlying set which is
continuous because the topology of $LRY$ is finer than the
original topology of $Y$.

% \medskip
\begin{definition}
  \label{dfn:standard_diffeology}
  The standard diffeology of a subset $X$ of $\R^k$ is
  defined by the formula
  \[
    X(U) = \{ f \colon U \to X \mid \text{the composition
      $U \xrightarrow{f} X \subset \R^k$ is smooth} \} \quad
    (U \in \EucOp).
  \]
  % which we call the standard diffeology of $X$.
\end{definition}
% \medskip

% The following is clear from the definition.
\begin{proposition}
  \label{prp:D-topology_is_metric_topology}
  If $X$ is an open subset of $\R^k$ then the $D$-topology
  associated with the standard diffeology of $X$ coincides
  with the subspace topology inherited from the standard
  metric topology on $\R^k$.  For general $X \subset \R^k$
  its $D$-topology is finer than the subspace topology.
\end{proposition}
% \medskip

\begin{example}
  % The $D$-topology of any interval in $\R$ is the subspace
  % topology inherited from the metric topology on $\R$, while
  % the $D$-topology of the set of rationals $\Q$ is not the
  % subspace topology but the discrete topology.
  The $D$-topology of the diffeological subspace of
  rationals $\Q \subset \R$ is the discrete topology, which
  is strictly finer than the subspace topology inherited
  from $\R$.
\end{example}
% \medskip

We summarize the basic features of $\Diff$ in the following.
% \medskip
\begin{theorem}
  \label{thm:category_of_diffeological_spaces}
  The category $\Diff$ enjoys the following properties:
  \begin{enumerate}
  \item $\Diff$ is closed under small limits and colimits.
  \item $\Diff$ is enriched over itself with $\Cinf(X,Y)$ as
    hom-object.
  \item $\Diff$ is cartesian closed with
    $\Cinf(X,Y)$ as exponential object.
  % \item There is a left adjoint functor
  %   $L \colon \Diff \to \Top$ extending the inclusion
  %   $\EucOp \to \Top$.
  \item $\Diff$ is related with $\Top$ via the adjunction
    $L \dashv R \colon \Diff \rightleftarrows \Top$.
  \end{enumerate}
\end{theorem}

%%%%%%%%%%%%%%%%%%%%
\subsection{Nonstandard number fields via ultrapower
  construction}
Let $\U$ be a free ultrafilter on the set of nonnegative
integers $\N$ that do not contain any finite subset.
Given a predicate $P(n)$ defined on $\N$, let us write
``$P(n)$ a.e.'' to mean $\{n \in \N \mid P(n)\} \in \U$.
Then there is a total preorder ``$\leq$'' on the algebra
$\R^{\N}$ defined by
\[
  (a_n) \leq (b_n)\ \ \text{if and only if}\ \ a_n \leq b_n\
  \aev
\]
and we define an ordered set $\nR$ to be the quotient of
$\R^{\N}$ by the equivalence relation:
\[
  (a_n) \sim (b_n) \ \ \text{if and only if} \ \ (a_n) \leq
  (b_n) \ \& \ (b_n) \leq (a_n).
\]
Equivalently, we can define $\nR = \R^{\N}/\nN(\R^{\N})$
where % $\nI(\R^{\N})$ is the maximal ideal
\[
  \nN(\R^{\N}) = \{(x_n) \in \R^{\N} \mid
  x_n = 0 \ \aev\}.
\]
Similarly, we define $\nC = \C^{\N}/\nN(\C^{\N})$ where
\[
  \nN(\C^{\N}) = \{(x_n) \in \C^{\N} \mid x_n = 0 \
  \aev\}.
\]
Members of $\nF$ (resp.\ $\nC$) are called nonstandard real
(resp.\ complex) numbers.  Nonstandard real numbers are also
called hyperreal numbers following A.~Robinson.
It is well known that with respect to the addition and
multiplication induced by the level-wise operations
% $\R^{\N} \times \R^{\N} \to \R^{\N}$,
$((a_n), (b_n)) \mapsto (a_n + b_n)$ or $(a_nb_n)$, $\nR$ is
a non-Archimedean real closed extension of $\R$ and $\nC$ is
an algebraically closed field of the form
$\nC = \nR + \sqrt{-1}\nR$.
% There exist ``absolute value'' functions
% $\nF \to \nR_{\geq 0}$, $x \mapsto \abs{x}$
% ($\F = \R,\, \C$).

%%%%%%%%%%
% \subsection{Quasi-asymptotic numbers}
Now, let $\rho \in \nR$ be a positive infinitesimal
represented by the net $(1/n) \in \R^{\N}$ and denote by
$\nst{\N}$ the set of {hypernatural} numbers, i.e.\ the
image of $\N^{\N}$ under the projection $\R^{\N} \to \nR$.
For $\F = \R,\,\C$ we define $\rF = \rM(\nF)/\rN(\nF)$,
where
\begin{align*}
  \rM(\nF)
  &= \{ x \in \nF \mid \exists c \in \nst{\N},\ \abs{x} \leq
    \rho^{-c}\}, \\
  \rN(\nF)
  &= \{ x \in \nF \mid \forall d \in \nst{\N},\ \abs{x} \leq
    \rho^{d} \}.
\end{align*}
Alternatively, we can define
$\rF = \rM(\F^{\N})/\rN(\F^{\N})$, where
\begin{align*}
  \rM(\F^{\N})
  &= \{ (x_n) \in \F^{\N} \mid \exists (c_n) \in
    \N^{\N},\ \abs{x_n} \leq
    n^{c_n} \ \aev \}, \\
  \rN(\F^{\N})
  &= \{(x_n) \in \F^{\N} \mid \forall (d_n) \in
    \N^{\N},\ \abs{x_n} \leq
    1/n^{d_n} \ \aev \}.
\end{align*}
Members of $\rF$ are called {quasi-asymptotic}
({real} or {complex}) numbers.

% \medskip
\begin{proposition}
  \label{prp:bF_is_closed_field}
  Let $\bullet$ be $*$ or $\rho$.  Then we have the
  following.
  \begin{enumerate}
  \item $\bR$ is a real closed, Cantor complete and
    non-Archimedean extension of\/ $\R$.
  \item $\bC$ is an algebraically closed field of the form\/
    $\bC = \bR + \sqrt{-1}\,\bR$.
  \end{enumerate}
\end{proposition}
% Here, ``Cantor complete'' means that every nested sequence
% of bounded closed intervals has a non-empty intersection.
% A brief proof of the proposition above will be given in
% the appendix.

\begin{proof}
  We only consider the case $\bullet = \rho$ as the case
  $\bullet = *$ is well known.  It is clear that $\rC$ is a
  ring and that we have $\rC = \rR + \sqrt{-1}\,\rR$.
  To see that $\rC$ is a field, suppose
  $(a_{n}) \in \rM(\C^{\N})$ represents a non-zero class in
  $\rC$.  Then there exist $(c_n),\, (d_n) \in \N^{\N}$ such
  that
  \[
    \Phi = \{ n \mid 1/n^{d_n} \leq \abs{a_n} \leq 1/n^{c_n}
    \} \in \U.
  \]
  %  (cf.\ Example~\ref{ex:infinite_product})
  Let us define $(b_n) \in \C^{\N}$ by $b_n = 1/a_n$ if
  $n \in \Phi$ and $b_n = 1$ if otherwise.  Then
  $(b_n) \in \rM(\C^{\N})$ because $\abs{b_n} \leq n^{d_n}$
  a.e., and we have $[b_n] = [a_n]^{-1}$ in $\rC$, implying
  $\rC$ is a field.
  To see that $\rC$ is algebraically closed, let
  \[
    P(x) = x^p + a_1x^{p-1} + \cdots + a_p \in \rC[x]
  \]
  % be a polynomial with coefficients in $\rC$,
  and choose representatives $(a_{k,n}) \in \rM(\C^{\N})$
  for $a_k \in \rC$.  For each $n \in \N$ let
  \[
    P_n(x) = x^p + a_{1,n}x^{p-1} + \cdots +
    a_{p,n}
  \]
  and take $x_n \in \C$ such that $P_n(x_n) = 0$.  Then
  $(x_n) \in \rM(\C^{\N})$ because we have
  $\abs{x_n} \leq 1 + \abs{a_{1,n}} + \cdots +
  \abs{a_{p,n}}$, and hence determines an element
  $[x_n] \in \rC$ that satisfies $P([x_n]) = 0$.  Thus $\rC$
  is algebraically closed, and consequently, its real part
  $\rR$ is a real closed field.
  % and is totally ordered in such a way that $a \geq 0$ if
  % and only if $a = b^2$ holds for some $b \in \rR$.
  Moreover, $\rR$ is non-Archimedean because it contains a
  non-zero infinitesimal $\rho = [1/n]$.
\end{proof}

\begin{remark}
  In many publications, the symbol $\rF$ denotes the field
  of {asymptotic numbers} defined as a quotient
  $\aM(\F^{\N})/\aM(\F^{\N})$, where
  % ($\F = \R,\, \C$)
  \begin{align*}
    \aM(\F^{\N})
    &= \{ (x_n) \in \F^{\N} \mid \exists c \in \N,\ \abs{x_n}
      \leq n^{c} \ \aev \}, \\
    \aN(\F^{\N})
    &= \{(x_n) \in \F^{\N} \mid \forall d \in \N,\ \abs{x_n}
      \leq 1/n^{d} \ \aev \}.
  \end{align*}
  In this paper, we denote $\aF = \aM(\F^{\N})/\aM(\F^{\N})$
  to avoid the conflict with our usage of $\rF$.  Again,
  $\aR$ is a non-Archimedean real closed field and $\aC$ is
  an algebraically closed field of the form
  $\aC = \aR + \sqrt{-1}\,\aR$.  We may regard $\aF$ as a
  subquotient of $\rF$ because we have
  $\rN(\F^{\N}) \subset \aN(\F^{\N}) \subset \aM(\F^{\N})
  \subset \rM(\F^{\N})$.
\end{remark}

%%%%%%%%%%%%%%%%%%%%%%%%%%%%%%%%%%%%%%%%%%%%%%%%%%%%%%%%%%%%
% Section 3
%%%%%%%%%%%%%%%%%%%%%%%%%%%%%%%%%%%%%%%%%%%%%%%%%%%%%%%%%%%%
\section{Differential calculus on nonstandard Euclidean
  spaces}
%%%%%%%%%%%%%%%%%%%%
\subsection{Internal topology of nonstandard Euclidean
  spaces}
Let $k \in \N$ and $\bullet$ be either $*$ or $\rho$.  By a
nonstandard Euclidean space of dimension $k$ we mean the
$k$-dimensional vector space $\bR^k$ over $\bR$ equipped
with a topology defined as in the following way.
For $k = 1$ and $\bullet = *$, we endow $\R^{\N}$ with the
box topology and consider $\nR$ as a quotient of $\R^{\N}$
by the projection $\R^{\N} \to \nR$.  %
Similarly, $\rR$ is defined as the quotient of the
projection $\rM(\R^{\N}) \to \rR$ where $\rM(\R^{\N})$ is
regarded as a topological subspace of $\R^{\N}$.  For
general $k$, $\bR^k$ is regarded as the $k$-fold product of
the topological field $\bR$.

To be more explicit, denote by $\mathcal{B}(\bR^k)$ the set
of internal subsets
\[ [V_n] = \{\, [x_n] \in \bR^k \mid (x_n) \in \tprod_{n \in
    \N} V_n \cap \bM(\R^{\N})^k \,\} \subset \bR^k
\]
where $V_n$ is open in $\R^k$ for every $n \in \N$ and we
put $\nM(\R^{\N}) = \R^{\N}$ in the case $\bullet = *$.
Then, By the definition of quotient topology, we have the
following.

% \medskip
\begin{lemma}
  % Let $k \geq 0$ and $\bullet = *$ or $\rho$.  Then
  $\mathcal{B}(\bR^k)$ is a basis for the topology of\/
  $\bR^k$, that is, $U$ is open in $\bR^k$ if and only if it
  is the union of members of $\mathcal{B}(\bR^k)$.
\end{lemma}
% \medskip

% \begin{proof}
%   As $L \colon \Diff \to \Top$ preserves colimits,
%   $L(\bR^k)$ has the quotient topology with respect to the
%   projection $L(\bM(\R^{\N})^k) \to L(\bR^k)$.  Thus $U$
%   is $D$-open if and only if so is its preimage
%   $U' \subset \bM(\R^k)$.
%   %
%   But as the topology of $(\R^{\N})^k = (\R^k)^{\N}$ is
%   the box topology,
%   %
%   $U'$ is open if and only if every its point is contained
%   in a neighborhood of the form
%   $\prod_{n \in \N} V_n \cap \bM(\R^{\N})^k$ with $V_n$ open
%   in $\R^k$.  This implies the statement for $U$.
% \end{proof}

It follows that the topology of $\bR^k$ is generated by
infinitesimal open sets, i.e.\ those open sets
$V \subset \bR^k$ such that $x \approx y$ for any
$x,\, y \in V$.

% \medskip
\begin{corollary}
  \label{crl:nonstandard_euclidean_space_is_hausdorff}
  The space $\bR^k$ is a Hausdorff topological vector space
  over $\bR$ containing\/ $\R^k$ as a discrete subset.
\end{corollary}
% \medskip

In particular, we can equip $\bC$ with a topology induced by
the bijection $\bC = \bR + \sqrt{-1}\bR \cong \bR^2$.

% \medskip
\begin{proposition}
  \label{prp:bF_is_a_topological_field}
  With respect to the topology defined above, $\bF$ is a
  topological field for $\F = \R,\, \C$ and
  $\bullet = *,\, \rho$.
\end{proposition}
% \medskip

\begin{definition}
  \label{dfn:dot_U}
  For any subset $M \subset \R^k$ denote by $\bas{\!M}$ the
  internal subset of $\bR^k$ given by the constant net
  $(M) \subset (\R^k)^{\N}$.  In particular, if $U$ is open
  in $\R^k$ then so is $\bas{U}$ in $\bR^k$.
\end{definition}

%%%%%%%%%%%%%%%%%%%%
\subsection{Differentiable functions on nonstandard
  Euclidean spaces}
\subsubsection{The case $\bullet = *$}
Let $U$ be an open subset of $\nR^k$ and let
$0 \leq r \leq \infty$.

% \medskip
\begin{definition}
  \label{dfn:nonstandard_differentiability}
  A function $f \colon U \to \nF$ is called a differentiable
  function of class $C^r$, or $C^r$-function for short, if
  for any $x \in U$ there exist an internal neighborhood
  $V = [V_n] \subset U$ of $x$ and a net of functions
  $(g_n \colon V_n \to \F)$ such that
  \begin{enumerate}[topsep=4pt]
  \item[(i)] each $g_n$ is of class $C^r$, and
  \item[(ii)] $f(y) = [g_n(y_n)]$ holds for every
    $y = [y_n] \in V$ with $y_n \in V_n$.
  \end{enumerate}
\end{definition}
% \medskip

Denote by $\nCinf[r](U,\nF)$ the set of all $C^r$-functions
$U \to \nF$.  Clearly, $\nCinf[r](U,\nF)$ is a commutative
algebra over $\nF$ under pointwise addition and
multiplication.  Moreover, there are $\nF$-linear operators
\[
  \partial_i \colon \nCinf[r+1](U,\nF) \to \nCinf[r](U,\nF)
  \quad (1 \leq i \leq k,\ r \geq 0)
\]
defined as in the following way:
Let $f \in \nCinf[r+1](U,\nF)$ and $x \in U$.  Then there
exist an internal open neighborhood $V = [V_n] \subset U$
and a net of $C^r$-functions $(g_n \colon V_n \to \F)$
satisfying (i) and (ii) of
Definition~\ref{dfn:nonstandard_differentiability} and we
can put
\begin{equation}
  \label{eq:derivative_of_hyperfunction}
  \partial_i{f(x)} = [\partial_i{g_n}(x_n)] \in \nF \quad (1
  \leq i \leq k)
\end{equation}
where $(x_n) \in \prod_{n \in \N} V_n$ is a net representing
$x$.  To see that the value of $\partial_i{f}(x)$ does not
depend on the choice of representatives $(g_n)$ and $(x_n)$,
suppose $(h_n \colon V'_n \to \F)$ is another net of smooth
functions such that $x \in V' = [V'_n]$ and
$f(y) = [h_n(y_n)]$ holds for $y = [y_n] \in V'$.  Let
$W_n = V_n \cap V'_n$ and put
$k_n = h_n - g_n \colon V_n \to \F \ (n \in \N)$ for
$n \in \N$.  Then the function $k \colon W = [W_n] \to \nF$
defined by $k([y_n]) = [k_n(y_n)]$ is constant with value
$0$.  But this implies
$[\partial_i g_n(x_n)] = [\partial_i h_n(x_n)]$ because of
the following.

% \medskip
\begin{lemma}
  If $k \colon W \to \nF$ is constant then we have
  $[\partial_i k_n(y_n)] = 0$ for all $[y_n] \in W$ and
  $1 \leq i \leq k$.
\end{lemma}

\begin{proof}
  Suppose that $[\partial_{i}k_n(y_n)] > 0$ holds for some
  $[y_n] \in W$.  Then we have $\partial_{i}k_n(y_n) > 0$
  a.e.\ and there exists a net $(z_n)$ with $z_n \in W_n$
  such that $k_n(z_n) > k_n(y_n)$ holds a.e.  But this means
  $[k_n(z_n)] > [k_n(y_n)]$, contradicting to the assumption
  that $k$ is constant.
  Similar contradiction occurs if we suppose
  $[\partial_{i}k_n(y_n)] < 0$, hence
  $[\partial_{i}k_n(y_n)] = 0$ holds everywhere.
\end{proof}

In particular, we have the following.

% \medskip
\begin{proposition}
  \label{prp:DA_of_hyperreal_functions}
  $\nCinf(U,\bF)$ is a differential algebra with respect to
  partial derivatives $\partial_i$ ($1 \leq i \leq k$), i.e.\
  there hold
  \begin{enumerate}[topsep=4pt]
  \item[{\rm (i)}]
    $\partial_i(fg) = (\partial_{i}f)g + f(\partial_{i}g)$
    for $1 \leq i \leq k$, and
  \item[{\rm (ii)}]
    $\partial_{i}\partial_{j}f = \partial_{j}\partial_{i}f$
    for $1 \leq i < j \leq k$.
  \end{enumerate}
\end{proposition}
% \medskip

% For any multi-index $\alpha = (a_1,\cdots,a_k) \in \N^k$
% denote by $D^{\alpha}$ the composite operator
% $\partial_1^{a_1} \cdots \partial_k^{a_k} \colon A \to A$.
% Thanks to the property (ii) above, we see that
% $D^{\alpha}D^{\beta} = D^{\alpha + \beta}$ holds for any
% $\alpha,\, \beta \in \N^k$.

For the extension $\nst{U} = [U] \subset \nR^k$ of an open
subset $U \subset \R^k$ there is an injection
\[
  i_U \colon \Cinf(U,\F) \to \nCinf(\nst{U},\nF),\ \ f
  \mapsto \nst{\!f},
\]
induced by the diagonal inclusion
$\Cinf(U,\F) \to \Cinf(U^{\N},\F^{\N})$, i.e.\
$\nst{\!f}(x) = [f(x_n)]$ for $x = [x_n] \in \nst{U}$.  It
is clear by the definition that the following holds.

% \medskip
\begin{proposition}
  \label{prp:inclusion_into_hyperreal_functions}
  For any open subset $U \subset \R^k$ the map
  $i_U \colon \Cinf(U,\F) \to \nCinf(\nst{U},\nF)$ is an
  inclusion of differential algebras.
\end{proposition}
% \medskip

The next proposition implies that the intermediate value
theorem (IVT) and the mean value theorem (MVT) hold at least
locally for any member of $\nCinf(U,\nF)$.

% \medskip
\begin{proposition}
  \label{prp:IVT_and_MVT}
  Let $U$ be an internal open subset of the form
  $U = [U_n] \subset \nR^k$ and $f \in \nCinf(U,\nR)$.
  Suppose $f$ is represented by a net of smooth functions
  $(f_n \colon U_n \to \R)$, that is, $f(z) = [f_n(z_n)]$
  holds for any $z = [z_n] \in U$.  Then the following hold
  for any $x,\, y \in U$ such that $(1-t)x+ty \in U$ for all
  $t \in \nst{[0,1]}$.
  \begin{description}[topsep=4pt]
  \item[\rm IVT:] If $f(x) \neq f(y)$ then for any $r \in \nF$
    between $f(x)$ and $f(y)$ there exists $z \in U$ such
    that $f(z) = r$ holds.
  \item[\rm MVT:]There exists $c \in \nst{(0,1)}$ such that
    $f(y)-f(x) = \nabla f((1-c)x+cy) \cdot (y-x)$ holds.
  \end{description}
  Here we denote by $\nabla$ the gradient
  $(\partial_1,\cdots,\partial_k)$ and by $\cdot$ the dot
  product.
\end{proposition}

\begin{proof}
  To prove IVT, assume $f(x) < f(y)$ and choose
  representatives $(x_n),\, (y_n)$ and $(r_n)$ for $x$, $y$
  and $r$, respectively.  Then $f_n(x_n) < r_n < f_n(y_n)$
  holds for almost every $n$ and there exists
  $(c_n) \in (0,1)^{\N}$ that satisfies
  % $z_n = (1-c_n)x_n+c_ny_n$ satisfies
  $f_n((1-c_n)x_n+c_ny_n) = r_n$ a.e.  Thus, by letting
  $c = [c_n]$ and $z = (1-c)x+cy \in U$ we have
  $f(z) = [f_n(z_n)] = [r_n] = r$ as desired.
  MVT is proved similarly.  For each $n \in \N$ choose
  $c_n \in (0,1)$ that satisfies
  \[
    f_n(y_n)-f_n(x_n) = \nabla f_n((1-c_n)x_n+c_ny_n) \cdot
    (y_n-x_n).
  \]
  Then $f(y)-f(x) = \nabla f((1-c)x+cy) \cdot (y-x)$ holds
  for $c = [c_n] \in \nst{(0,1)}$.
\end{proof}

%%%%%%%%%%
\subsubsection{The case $\bullet = \rho$}
Given an open subset $U \subset \rR^k$ let
$U' = (q^k)^{-1}(U) \subset \rM(\nR)^k$, where $q^k$ is the
projection $\rM(\nR)^k \to \rR^k$.
Since $\rM(\nR)$ is the image of a linear subspace
$\rM(\R^{\N})$ of $\R^{\N}$, we can define partial
derivatives of a function $U' \to \nF$ by the formula
similar to \eqref{eq:derivative_of_hyperfunction}.  Thus we
have a set $\nCinf[r](U',\nF)$ of differentiable functions
$U' \to \nF$ of class $C^r$ ($1 \leq r \leq \infty$).  In
particular, $\nCinf(U',\nF)$ is a differential algebra over
$\nF$.
We now define for $\mathbf{L} = \mathbf{M},\, \mathbf{N}$ a
subalgebra $\rL(\nCinf(U',\nF)) \subset \nCinf(U',\nF)$ as
follows:
\[
  \rL(\nCinf(U',\nF)) := \{\, \phi \in \nCinf(U',\nF) \mid
  \forall \alpha \in \N^k,\ \forall y \in U',\
  D^{\alpha}\phi(y) \in \rL(\nF) \,\}.
\]
Here $D^{\alpha}$ denotes the composite operator
$\partial_1^{a_1} \cdots \partial_k^{a_k}$ for
$\alpha = (\alpha_1,\cdots,\alpha_k) \in \N^k$.
% Note that $D^{\alpha}D^{\beta} = D^{\alpha + \beta}$ holds
% for any $\alpha,\, \beta \in \N^k$.
Clearly, $\rN(\nCinf(U',\nF))$ is a subalgebra of
$\rM(\nCinf(U',\nF))$ and we put
\[
  \rCinf(U,\rF) = \rM(\nCinf(U',\nF))/\rN(\nCinf(U',\nF)).
\]
The proposition below indicates that members of
$\rCinf(U,\rF)$ can be considered as $\rF$-valued infinitely
differentiable functions defined on $U$.

% \medskip
\begin{proposition}
  \label{prp:DA_of_quasi_asymptotic_functions}
  Let $U$ be an open subset of $\rR^k$.  Then the following
  hold.
  \begin{enumerate}
  \item There is an inclusion
    $\rCinf(U,\rF) \to \hom_{\Set}(U,\rF)$.
  \item $\rCinf(U,\rF)$ is a differentiable algebra over
    $\rF$.
  \item If $U$ is an open subset of $\R^k$ then there is an
    injection of differentiable algebras
    $i_U \colon \Cinf(U,\F) \to \rCinf(\ras{U},\rF)$.
  \end{enumerate}
\end{proposition}

\begin{proof}
  For each $f \in \rCinf(U,\rF)$ and $\alpha \in \N^k$
  define a function $D^{\alpha}f \colon U \to \rF$ by
  \begin{equation}
    \label{eq:derivatives_of_quasi_asymptotic_function}
    D^{\alpha}f(x) = [D^{\alpha}\phi(y)] \in \rF \quad (x
    \in U)
  \end{equation}
  where $\phi \in \rM(\nCinf(U',\nF))$ is a representative
  of $f$ and $y \in (q^k)^{-1}(x)$.  To see that
  \eqref{eq:derivatives_of_quasi_asymptotic_function} is
  well defined, let $\psi \in \rM(\nCinf(U',\nF))$ be
  another representative of $f$ and $z \in (q^k)^{-1}(x)$.
  Then we have
  \[
    D^{\alpha}\phi(y) - D^{\alpha}\psi(z) =
    D^{\alpha}(\phi-\psi)(y) + D^{\alpha}\psi(y) -
    D^{\alpha}\psi(z).
  \]
  Since $\phi-\psi \in \rN(\nCinf(U',\nF))$, we have
  $D^{\alpha}(\phi-\psi)(y) \in \rN(\nF)$.  On the other
  hand, because $(1-c)y+cz \in (q^k)^{-1}(x)$ for
  $c \in \nst{(0,1)}$ and $y-z \in \rN(\nF)^k$, we have
  $D^{\alpha}\psi(y) - D^{\alpha}\psi \circ q^k)(z) \in
  \rN(\nF)$ by applying the MVT to
  $D^{\alpha}\psi \in \nCinf(U',\nF)$.  Thus we have
  $D^{\alpha}\phi(y) - D^{\alpha}\psi(z) \in \rN(\nF)$,
  implying
  \eqref{eq:derivatives_of_quasi_asymptotic_function} is
  well defined.
  In particular, the correspondence $f \mapsto D^0{f}$
  provides an inclusion
  $\rCinf(U,\rF) \to \hom_{\Set}(U,\rF)$ and $\rCinf(U,\rF)$
  is a differential algebra with respect to differential
  operators $f \mapsto D^{\alpha}f$, proving (1) and (2).

  To prove (3), consider the inclusion
  $i_U \colon \Cinf(U,\F) \to \nCinf(\nst{U},\nF)$ which
  takes $f \in \Cinf(U,\F)$ to the function $\nst{\!f}$
  defined by $\nst{\!f}([x_n]) = [f(x_n)]$.
  But its derivatives $D^{\alpha}(\nst{\!f})$ have values in
  $\rM(\nF)$ because for any $(x_n) \in U^{\N}$ there is a
  net $(c_n) \in \N^{\N}$ satisfying
  $\abs{D^{\alpha}f(x_n)} \leq n^{c_n}$ for $n \geq 2$.
  Hence $\nst{\!f} \in \rM(\nCinf(\nst{U},\nF))$ holds for
  every $f \in \Cinf(U,\F)$ and we can define an injection
  $i_U \colon \Cinf(U,\F) \to \rCinf(\ras{U},\rF)$ as the
  composition
  $\Cinf(U,\F) \to \rM(\nCinf(\nst{U},\nF)) \to
  \rCinf(\ras{U},\rF)$.
\end{proof}

The IVT and the MVT also hold for quasi-asymptotic
functions.

% \medskip
\begin{proposition}
  \label{prp:asymptotic_version_of_IVT_and_MVT}
  Let $U$ be an internal open subset of the form
  $U = [U_n] \subset \rR^k$ and $f \in \rCinf(U,\rR)$.
  Suppose $f$ is represented by a net of smooth functions
  $(f_n \colon U_n \to \R)$.  Then the following hold for
  any $x,\, y \in U$ such that $(1-t)x+ty \in U$ for
  $t \in \ras{[0,1]}$.
  \begin{description}[topsep=4pt]
  \item[\rm IVT:] If $f(x) \neq f(y)$ then for any
    $r \in \rF$ between $f(x)$ and $f(y)$ there exists
    $z \in U$ such that $f(z) = r$ holds.
  \item[\rm MVT:] There exists $c \in \ras{(0,1)}$ such that
    $f(y)-f(x) = \nabla f((1-c)x+cy) \cdot (y-x)$ holds.
  \end{description}
  Here we denote by $\nabla$ the gradient
  $(\partial_1,\cdots,\partial_k)$ and by $\cdot$ the dot
  product.
\end{proposition}
% \medskip

%%%%%%%%%%%%%%%%%%%%
\subsection{Infinitely differentiable maps between open
  subsets of nonstandard Euclidean spaces}
Given open subsets $U \subset \bR^k$ and $V \subset \bR^l$,
let $\bCinf[r](U,V)$ denote the set of maps
$f \colon U \to V$ whose component functions
$f_1,\, \cdots,\, f_l \colon U \to \bR$ belong to
$\bCinf[r](U,\bR)$, i.e.\
\[
  % \bCinf(U,V) := \bCinf(U,\bR)^l \cap \hom_{\Set}(U,V)
  % \subset \hom_{\Set}(U,\bR^l).
  \bCinf[r](U,V) := \{ f = (f_1,\, \cdots,\, f_l) \colon U
  \to \bR^l \mid f_1,\, \cdots,\, f_l \in
  \bCinf[r](U,\bR),\ \Image{f} \subset V \}.
\]
Members of $\bCinf(U,V)$ are called infinitely
differentiable maps from $U$ to $V$.
% Notice that smoothness in the sense of diffeology implies
% infinite differentiability in the case $\bullet = *$, but
% not necessarily in the case $\bullet = \rho$.

% \medskip
\begin{proposition}
  \label{prp:composition_of_nonstandard_maps}
  Suppose $U \subset \bR^k$, $V \subset \bR^l$ and
  $W \subset \bR^m$ are open subsets.  Then the composition
  of maps % $C(V,W) \times C(U,V) \to C(U,W)$
  induces a pairing\/
  $\bCinf(V,W) \times \bCinf(U,V) \xrightarrow{\circ}
  \bCinf(U,W)$ which is subject to the chain rule:
  \begin{equation}
    \label{eq:chain_rule}
    J_{g \circ f}(x) = J_{g}(f(x))\,J_{f}(x) \quad (f \in
    \bCinf(U,V),\ g \in \bCinf(V,W),\ x \in U)
  \end{equation}
  where $J_{f}(x) = (\partial_jf_i(x))_{i,j}$ is the
  Jacobian matrix of $f = (f_1,\cdots,f_l)$ at $x$ and
  similarly for $J_{g}$ and $J_{g \circ f}$.
\end{proposition}

\begin{proof}
  The case $\bullet = *$ is immediate from the definition of
  differential operators.
  To prove the case $\bullet = \rho$, write
  $U' = (q^k)^{-1}(U)$, $V' = (q^l)^{-1}(V)$ and
  $W' = (q^m)^{-1}(W)$, where $q$ is the projection
  $\rM(\nR) \to \rR$, and let
  \[
    \rL(U',V') := \rL(\nCinf(U',\nR))^l \cap \nCinf(U',V')
    \subset \nCinf(U',\nR)^l %
    \quad (\mathbf{L = M,\, N}).
  \]
  Then the product of the projections
  $\rM(\nCinf(U',\nR))^l \to \rCinf(U,\rR)^l$ restricts to a
  map $\Psi \colon \rM(U',V') \to \rCinf(U,V)$ and we have a
  diagram:
  \begin{equation}
    \label{eq:composition_of_smooth_maps}
    \vcenter{%
      \xymatrix{%
        \nCinf(V',W') \times \nCinf(U',V') \ar[r]^-{\circ}
        & \nCinf(U',W')
        \\
        \rM(V',W') \times \rM(U',V') \ar[r]^-{\circ}
        \ar[u]^-{\cup} \ar[d]_-{\Psi \times \Psi}
        & \rM(U',W') \ar[u]_-{\cup} \ar[d]^-{\Psi}
        \\
        \rCinf(V,W) \times \rCinf(U,V) \ar[r]^-{\circ}
        & \rCinf(U,W).
      }%
    }%
  \end{equation}
  The chain rule for the case $\bullet = *$ implies that the
  upper square of \eqref{eq:composition_of_smooth_maps} is
  commutative and also that
  $\rM(V',W') \times \rM(U',V') \xrightarrow{\circ}
  \rM(U',W')$ restricts to
  \[
    \rN(V',W') \times \rM(U',V') \cup \rM(V',W') \times
    \rN(U',V') \xrightarrow{\circ} \rN(U',W').
  \]
  Thus there is a map
  $\rCinf(V,W) \times \rCinf(U,V) \xrightarrow{\circ}
  \rCinf(U,W)$ which makes the diagram
  \eqref{eq:composition_of_smooth_maps} commutative and is
  subject to the chain rule \eqref{eq:chain_rule}.
\end{proof}

% The proposition above implies that nonstandard Euclidean
% open sets and infinitely differentiable maps form a
% subcategory $\bEucOp$ of $\Diff$.  Note that $\nEucOp$ is a
% full subcategory, but $\rEucOp$ is not.
The proposition above says that nonstandard Euclidean open
sets and infinitely differentiable maps form a category
which we denote $\bEucOp$.

% \medskip
\begin{theorem}
  \label{thm:nonstandard_Euclidean_open_sets_form_a_site}
  The category\/ $\bEucOp$ is a concrete site equipped with
  the coverage consisting of open covers.  Moreover, there
  is a faithful functor $i \colon \EucOp \to \bEucOp$ that
  preserves covering families.
\end{theorem}

\begin{proof}
  Let $\{U_i \to U\}$ be a cover of $U \in \bEucOp$ by its
  open subsets, and $g \colon V \to U$ be a morphism in
  $\bEucOp$.  Then we can define a cover $\{V_j \to V\}$ of
  $V$ by $V_j = g^{-1}(U_j)$ and obtain a commutative square
  \[
    \xymatrix{%
      V_j \ar[d]_-{\cap} \ar[r]^-{g|V_j}
      & U_j \ar[d]^-{\cap}
      \\
      V \ar[r]^-{g} & U }%
  \]
  Hence the function which assigns to each $U \in \bEucOp$
  the collection of open covers of $U$ defines a
  coverage on $\bEucOp$.

  To construct $i \colon \EucOp \to \bEucOp$, assign to any
  $f \in \Cinf(U,V)$ with $U \subset \R^k$ and
  $V \subset \R^l$ a morphism
  $\bas{\!f} \in \bCinf(\bas{U},\bas{V})$ defined by
  $\nst{\!f}([x_n]) = [f(x_n)] \in \nst{V}$ for
  $[x_n] \in \nst{U}$, and
  $\ras{\!f} = \Psi(\nst{\!f}|\rM(\ras{U}',\ras{V}')) \in
  \rCinf(\ras{U},\ras{V})$.
  % It is now evident from
  % Proposition~\ref{prp:composition_of_nonstandard_maps}
  % that
  Then the correspondence which takes $U \in \EucOp$ to
  $\bas{U} \in \bEucOp$ and $f \in \Cinf(U.V)$ to
  $\bas{\!f} \in \bCinf(\bas{U},\bas{V})$ gives a faithful
  functor $i \colon \EucOp \to \bEucOp$ that preserves
  covering families.
\end{proof}

Note that
$i \colon \Cinf(U,V) \to \bCinf(\bas{U},\bas{V})$ is the
previously defined $i_U$ in the case $V = \R$.

%%%%%%%%%%%%%%%%%%%%%%%%%%%%%%%%%%%%%%%%%%%%%%%%%%%%%%%%%%%%
% Section 4
%%%%%%%%%%%%%%%%%%%%%%%%%%%%%%%%%%%%%%%%%%%%%%%%%%%%%%%%%%%%
\section{The category of nonstandard diffeological spaces}
%%%%%
% \subsection{$\bullet$-Diffeological spaces as sheaves on
%   $\bEucOp$}%
\subsection{Nonstandard diffeological spaces as sheaves on
  $\bEucOp$}%
Recall that diffeological spaces can be defined as the
concrete sheaves on the site $\EucOp$.  This leads us to the
definition of extended diffeological spaces below.  % \medskip

\begin{definition}
  \label{dfn:nonstandard_diffeological_spaces}
  For $\bullet = *,\, \rho$ we denote by $\bDiff$ the
  category of concrete sheaves on $\bEucOp$.
  Objects of $\bDiff$ are called nonstandard diffeological
  spaces and their morphisms are called nonstandard smooth
  maps.  More specifically, we use the term hyper instead of
  nonstandard in the case $\bullet = *$ and quasi-asymptotic
  in the case $\bullet = \rho$.
\end{definition}
% \medskip

As in the case of usual diffeological spaces, nonstandard
diffeological spaces can be interpreted in terms of plots
defined on open subsets of nonstandard Euclidean spaces.
% \medskip

\begin{example}[Standard $\bullet$-diffeology]
  \label{ex:standard_nonstandard_diffeology}
  The standard $\bullet$-diffeology of a subset $X$ of
  $\bR^k$ is defined by the formula
  \[
    X(U) = \{ f \colon U \to X \mid \text{the composition
      $U \xrightarrow{f} X \subset \bR^k$ is smooth} \}
    % \quad (U \in \bEucOp).
  \]
  for $U \in \bEucOp$.  In particular, $\bR$ is a
  nonstandard diffeological space with respect to the
  standard $\bullet$-diffeology, and so is $\bC$ under the
  isomorphism $\bC \cong \bR^2$.
\end{example}
% \medskip

% Then $LX$ has the box topology generated by the family
% $\{ \prod_{j \in J} U_j \mid U_j\ \text{open in}\ LX_j \}$,
% hence $LX$ is not homeomorphic to the topological product
% $\prod_{j \in J}LX_j$ unless $LX_J$ is nontrivial for only
% finitely many $j$.
% %
% Notice that the left adjoint functor
% $L \colon \Diff \to \Top$ preserves colimits but not
% necessarily limits.

%%%%%%%%%%%%%%%%%%%%
% Let us regard $\F^{\N}$ as the product of countably many
% copies of $\F$ equipped with the standard diffeology.
% %
% Then we can endow $\nF$ with the quotient diffeology by the
% projection $q \colon \F^{\N} \to \nF$, and consequently its
% subquotient $\rF = \rM(\nF)/\rN(\nF)$.  With respect to
% these diffeological structures we can prove the following.

Recall that $\F$ is a field object in the sense that its
algebraic operations of addition, subtraction,
multiplication and division by non-zero elements are
morphisms (i.e.\ smooth maps) in $\Diff$.  This property
propagates to the nonstandard field $\bF$.

% \medskip
\begin{proposition}
  \label{prp:bF_is_smooth_field}
  % For both $\nF$ and $\rF$, their arithmetic operations of
  % addition, subtraction, multiplication and division by
  % non-zero objects are morphisms in $\bDiff$.
  The arithmetic operations of addition, subtraction,
  multiplication and division by non-zero objects of the
  field\/ $\bF$ are morphisms in $\bDiff$.
\end{proposition}
% \medskip

For any $X,\, Y \in \bDiff$ denote by $\hom(X,Y)$ the set of
morphisms (i.e.\ nonstandard smooth maps) from $X$ to $Y$
and define $\bCinf(X,Y) \in \bDiff$ by the formula:
\[
  \bCinf(X,Y)(U) = \{ \sigma \colon U \to \hom(X,Y) \mid
  \forall \tau \in X(U),\ (u \mapsto \sigma(u)(\tau(u)) \in
  Y(U) \}.
\]
% for every $U \in \bEucOp$.
Observe that $\bCinf(X,Y)$ has the underlying set
\[
  \abs{\bCinf(X,Y)} = \bCinf(X,Y)(0) = \hom(X,Y).
\]

% \medskip
% \begin{theorem}
%   \label{thm:category_of_nonstandard_diffeological_spaces}
%   The category $\bDiff$ is
%   \begin{enumerate}
%   \item closed under small limits and colimits,
%   \item enriched over itself with $\bCinf(X,Y)$ as
%     hom-objects, and
%   \item cartesian closed with $\bCinf(X,Y)$ as
%     exponential objects.
%   \end{enumerate}
%   Moreover, there is an adjunction
%   $\tr \dashv \dg \colon \Diff \rightleftarrows \bDiff$ such
%   that
%   \begin{enumerate}[resume]
%   \item $\tr$ extends $i \colon \EucOp \to \bEucOp$ and
%     preserves small colimits and finite limits, and
%   \item $\dg$ extends the inclusion $\bEucOp \to \Diff$ and
%     preserves small limits.
%   \end{enumerate}
% \end{theorem}
\begin{theorem}
  \label{thm:category_of_nonstandard_diffeological_spaces}
  The category $\bDiff$ is
  \begin{enumerate}
  \item closed under small limits and colimits,
  \item enriched over itself with $\bCinf(X,Y)$ as
    hom-objects, and
  \item cartesian closed with $\bCinf(X,Y)$ as
    exponential objects.
  \end{enumerate}
  Moreover, there is a functor $\tr \colon \Diff \to \bDiff$
  which extends $i \colon \EucOp \to \bEucOp$ and has a
  right adjoint $\dg \colon \bDiff \to \Diff$.
\end{theorem}

\begin{proof}
  (1) A category is closed under small limits if it has
  equalizers and small products, and is closed under small
  colimits if it has coequalizers and small coproducts.
  Thus the statement is a consequence of the constructions
  below.  % \medskip

  \begin{itemize}
  \item The product $\prod_{j \in J}X_j$ of $X_j \in \bDiff$
    ($j \in J$) is the sheaf
    \[
      U \mapsto \tprod_{j \in J}X_j(U) \quad (U \in
      \bEucOp).
    \]
    % with projections $\prod_{j \in J}X_j \to X_j$ given by
    % level-wise projections.
  \item The coproduct $\coprod_{j \in J}X_j$ of
    $X_j \in \bDiff$ ($j \in J$) is the sheafification of
    the presheaf
    \[
      U \mapsto \tcoprod_{j \in J} X_j(U) \quad (U \in
      \bEucOp).
    \]
    % More explicitly,
    % $\sigma \colon U \to \coprod_{j \in J}\abs{X_j}$ is a
    % plot of $\coprod_{j \in J}X_j$ if every $x \in U$ has
    % a neighborhood $V \subset U$ and a plot
    % $\tau \in X_j(V)$ such that $\tau = \sigma|V$.  The
    % inclusion $X_j \to \coprod_{j \in J}X_j$ is given by
    % the maps $\sigma \mapsto i_j \circ \sigma$ for
    % $\sigma \in X_j(U)$.
  \item The equalizer $\Eq(f,g)$ of a pair of smooth maps
    $f,\, g \colon X \to Y$ is a subsheaf of $X$,
    \[
      U \mapsto \Eq(f,g)(U) = \{ \sigma \in X(U) \mid f
      \circ \sigma = g \circ \sigma \} \quad (U \in
      \bEucOp).
    \]
  \item The coequalizer $\Coeq(f,g)$ of
    $f,\, g \colon X \to Y$ is the sheafification of the
    presheaf
    \[
      U \mapsto \Coeq(f,g)(U) = \{ p \circ \sigma \mid
      \sigma \in Y(U) \} \quad (U \in \bEucOp)
    \]
    where $p$ is the projection of $\abs{Y}$ onto its
    quotient by the minimal equivalence relation such that
    $f(x) \sim g(x)$ for every $x \in \abs{X}$.
  \end{itemize}
  % \medskip

  (2) We need to show that the composition
  \( \bCinf(Y,Z) \times \bCinf(X,Y) \xrightarrow{\circ}
  \bCinf(X,Z) \) is a morphism in $\bDiff$.  Let
  $(\tau,\sigma) \in \bCinf(Y,Z)(U) \times \bCinf(X,Y)(U)$
  and $\eta \in X(U)$.  Then
  $u \mapsto \sigma(\eta(u)) \in Y(U)$ and consequently,
  \[
    u \mapsto \tau(\sigma(\eta(u))) = (\tau \circ
    \sigma)(\eta(u)) \in Z(U)
  \]
  implying that $\tau \circ \sigma \in \bCinf(X,Z)(U)$ as
  desired.
  
  (3) Let
  $\gamma \colon \hom(X \times Y,Z) \to \hom(X,\bCinf(Y,Z))$
  be the set map which takes a morphism
  $f \colon X \times Y \to Z$ to
  $\gamma(f) \colon X \to \bCinf(Y,Z)$ given by the
  formula:
  \[
    \gamma(f)(\tau(u))(\eta(u)) = f(\tau(u),\eta(u)) \quad
    (\tau \in X(U),\, \eta \in Y(U),\, u \in U \in \bEucOp).
  \]
  To prove the cartesian closedness of $\bDiff$, it suffices
  to show that $\gamma$ induces an isomorphism
  $\bCinf(X \times Y,Z) \xrightarrow{\cong}
  \bCinf(X,\bCinf(Y,Z))$.
  Let $\sigma \colon U \to \hom(X \times Y,Z)$ be a plot of
  $\bCinf(X \times Y,Z)$.  Then the composition
  $\gamma_*(\sigma) = \gamma \circ \sigma$ is a plot of
  $\bCinf(X,\bCinf(Y,Z))$ because for any $\tau \in X(U)$,
  $\eta \in Y(U)$ and $u \in U$ we have
  \[
    \gamma(\sigma(u))(\tau(u))(\eta(u)) =
    \sigma(u)(\tau(u),\eta(u)) \in Z(U).
  \]
  Thus we obtain a morphism
  $\gamma_* \colon \bCinf(X \times Y,Z) \to
  \bCinf(X,\bCinf(Y,Z))$ in $\bDiff$.  That $\gamma_*$ is an
  isomorphism follows from the fact that $\gamma$ has an
  inverse $\gamma^{-1}$
  % \colon \hom(X,\bCinf(Y,Z)) \to \hom(X \times Y,Z)$
  which takes $g \in \hom(X,\bCinf(Y,Z))$ to $\gamma^{-1}(g)
  \in \hom(X \times Y,Z)$ given by the
  formula:
  \[
    \gamma^{-1}(g)(\tau(u),\eta(u)) = g(\tau(u))(\eta(u))
    \quad ((\tau,\eta) \in X(U) \times Y(U),\, u \in U \in
    \bEucOp).
  \]

  We now construct an adjunction $\tr \dashv \dg \colon
  \Diff \rightleftarrows \bDiff$.  Given $X \in
  \Diff$, let $\bas{\!X}$ denote the object of
  $\bDiff$ having the underlying set
  \[
    \textstyle %
    \abs{\bas{\!X}} := \underset{V \in \EucOp}{\coprod}\,
    \bas{V} \times X(V)/\!\sim
  \]
  where $(v,\sigma) \in \bas{V} \times
  X(V)$ is identified with $(w,\tau) \in \bas{W} \times
  X(W)$ if there is a smooth map $\phi \in
  \Cinf(V,W)$ such that $\sigma = \tau \circ
  \phi$ and $\bas{\!\phi}(v) =
  w$, and whose plots are set maps locally of the form
  \begin{equation}
    \label{eq:plots_of_tr}
    U \xrightarrow{\phi} \bas{W} \xrightarrow{\bas{\sigma}}
    \abs{\bas{\!X}} \quad (U \in \bEucOp,\ \phi \in
    \bCinf(U,\bas{W}),\ \sigma \in X(W))
  \end{equation}
  where $\bas{\sigma}$ denotes the composition $\bas{W} =
  \bas{W} \times \{\sigma\} \subset \bas{W} \times X(W) \to
  \abs{\bas{\!X}}$.
  Then we can assign to every $f \in
  \Cinf(X,Y)$ a morphism $\bas{\!f} \in
  \bCinf(\bas{\!X},\bas{Y})$ by the formula:
  $\bas{\!f}([u,\sigma]) = [u,f \circ \sigma] \in
  \abs{\bas{Y}}$.
  Thus, the correspondence $X \mapsto
  \bas{\!X}$ defines a functor $\tr \colon \Diff \to
  \bDiff$, whose restriction to
  $\EucOp$ is naturally isomorphic to $i \colon \EucOp \to
  \bEucOp$ because for every $V \in \EucOp$ the map $\tr(V)
  \to \bas{V}$ induced by the maps $\bas{U} \times V(U) =
  \bas{U} \times \Cinf(U,V) \to \bas{V}$, $(u,\sigma)
  \mapsto \bas{\sigma}(u)$, has an inverse $\bas{V} \to
  \tr(V)$ which takes $v \in
  \bas{V}$ to the class of $(v,1_V) \in \bas{V} \times
  V(V)$.

  On the other hand, $\dg \colon \bDiff \to \Diff$ is
  defined as a functor such that for any $Y \in \bDiff$,
  $\dg(Y)$ has the same underlying set as $Y$ and the plots
  locally of the form
  \begin{equation}
    \label{eq:plots_of_dg}
    U \xrightarrow{\subset} \bas{U} \xrightarrow{\tau}
    \abs{Y} \quad (U \in \EucOp,\ \tau \in Y(\bas{U})).
  \end{equation}
  For any $X \in \Diff$ define
  $\eta \colon \abs{X} \to \abs{\dg(\bas{\!X})} =
  \abs{\bas{\!X}}$ to be the set map
  \[
    \textstyle %
    \abs{X} = %
    \underset{U \in \EucOp}{\coprod}\, U \times X(U)/\!\sim
    % \ \ \overset{\subset}{\longrightarrow} %
    \ \ \to \ %
    \underset{U \in \EucOp}{\coprod}\, \bas{U} \times
    X(U)/\!\sim \ = \abs{\bas{\!X}}
  \]
  induced by the inclusions
  $U \times \{\sigma\} \subset \bas{U} \times \{\sigma\}$
  for $\sigma \in X(U)$.  Then $\eta$ is a morphism in
  $\Diff$ because the composition
  $\eta \circ \sigma \colon U \to \abs{\bas{\!X}}$ is a plot
  of $\dg(\bas{\!X})$ for any $\sigma \in X(U)$ as we see
  from \eqref{eq:plots_of_dg} by taking the inclusion
  $\bas{U} \times \{\sigma\} \subset \abs{\bas{\!X}}$ as
  $\tau$.  Hence we obtain a natural transformation
  $\eta \colon 1 \to \dg \circ \tr$.

  On the other hand, for any $Y \in \bDiff$ we can define
  $\varepsilon \colon \abs{\bas{\dg(Y)}} \to \abs{Y}$ as
  follows.  Let $[u,\sigma] \in \abs{\bas{\dg(Y)}}$ be a
  point represented by
  $(u,\sigma) \in \bas{U} \times \dg(Y)(U)$.  We may suppose
  without loss of generality that $\sigma$ is a composition
  $U \xrightarrow{\subset} \bas{U} \xrightarrow{\tau}
  \abs{Y}$ with $\tau \in Y(\bas{U})$,
  % (cf.\ \eqref{eq:plots_of_dg})
  and we can put
  $\varepsilon([u,\sigma]) = \tau(u) \in \abs{Y}$.  This
  does not depend on the choice of representative
  $(u,\sigma)$ and determines a map
  $\varepsilon \colon \abs{\bas{\dg(Y)}} \to \abs{Y}$.
  To see that $\varepsilon$ is a morphism in $\bDiff$,
  consider a generating plot
  $U \xrightarrow{\phi} \bas{W} \xrightarrow{\bas{\sigma}}
  \abs{\bas{\dg(Y)}}$ of $\bas{\dg(Y)}$, where
  $U \in \bEucOp$, $W \in \EucOp$,
  $\phi \in \bCinf(U,\bas{W})$, and $\sigma \in \dg(Y)(W)$
  is of the form
  $W \xrightarrow{\subset} \bas{W} \xrightarrow{\tau}
  \abs{Y}$.  Then the composition
  $U \xrightarrow{\phi} \bas{W} \xrightarrow{\bas{\sigma}}
  \abs{\bas{\dg(Y)}} \xrightarrow{\varepsilon} \abs{Y}$ is a
  plot of $Y$ because it coincides with the plot
  $U \xrightarrow{\phi} \bas{W} \xrightarrow{\tau} \abs{Y}$
  of $Y$, implying $\varepsilon$ is a morphism in $\bDiff$.
  Thus we have a natural transformation
  $\varepsilon \colon \tr \circ \dg \to 1$.

  That $\dg$ is a right adjoint of $\tr$ follows from the
  fact that the compositions
  \[
    \tr \xrightarrow{\tr \circ \eta} \tr \circ \dg \circ \tr
    \xrightarrow{\varepsilon \circ \dg} \tr %
    \ \ \text{and} \ \ %
    \dg \xrightarrow{\eta \circ \dg} \dg \circ \tr \circ \dg
    \xrightarrow{\dg \circ \varepsilon} \dg
  \]
  are the identity natural transformations of $\tr$ and
  $\dg$ respectively.
\end{proof}

\begin{remark}
  It follows by the property of adjunction that the left
  adjoint functor $\tr$ preserves small colimits and its
  right adjoint $\dg$ preserves small limits.  Additionally,
  $\tr$ preserves finite limits by the commutativity of
  filtered colimits with finite limits.
\end{remark}
% \medskip

In addition to the adjunction
$\tr \dashv \dg \colon \Diff \rightleftarrows \bDiff$ stated
above, there is an adjunction
$\bas{\!L} \dashv \bas{\!R} \colon \bDiff \rightleftarrows
\Top$, where $\bas{\!L}$ takes a nonstandard diffeological
space $X$ to its underlying set equipped with the initial
topology with respect to the plots of $X$, called the
$\bD$-topology of $X$, and $\bas{\!R}$ takes a topological
space $Y$ to the object $\bas{\!R}(Y) \in \bDiff$ such that
$\bas{\!R}(Y)(U)$ consists of all the continuous maps
$U \to Y$.
Since every plot $\sigma \colon U \to X$ of $X \in \bDiff$
determines a continuous map $U \to \bas{\!L}(X)(U)$, there
is a natural inclusion
$X(U) \subset \bas{\!R}\,\bas{\!L}(X)(U)$ providing the unit
of the adjunction
$\eta \colon X \to \bas{\!R}\,\bas{\!L}(X)$.
On the other hand, the counit of the adjunction
$\varepsilon \colon \bas{\!L}\,\bas{\!R}(Y) \to Y$ is given
by the identity of the underlying set which is continuous
because the topology of $\bas{\!L}\,\bas{\!R}(Y)$ is finer
than the original topology of $Y$.

% \medskip
\begin{proposition}
  \label{prp:properties_of_L}
  If $X$ is an open subset of\, $\bR^n$ considered as an
  object of $\bDiff$ then its $\bD$-topology is the subspace
  topology induced from the internal topology on\/ $\bR^n$.
\end{proposition}

\begin{proof}
  % Let $W$ be a subset of the underlying set of $X$.  Since
  % $\abs{X}$ is open in $\bR^n$, $W$ is open in $\abs{X}$ if
  % and only if so is in $\bR^n$.  Hence $W$ is open in
  % $\abs{X}$ implies $\sigma^{-1}(W)$ is open in $U$ for
  % every plot $\sigma \in X(U)$.  The converse implication is
  % proved by taking the identity of $X \subset \bR^n$ as a
  % plot of $X$.
  This is a consequence of the fact that the
  $\bullet$-diffeology of $X$ is generated by the inclusion
  of $X$ as an open subset of $\bR^n$ and the $\bD$-topology
  of $\bR^n$ is nothing but the internal topology of
  $\bR^n$.
\end{proof}

\begin{proposition}
  \label{prp:natural_extension_is_continuous}
  There are natural transformations
  \[
    L \circ \dg \to \bas{\!L} \colon \bDiff \to \Top, \quad
    L \to \bas{\!L} \circ \tr \colon \Diff \to \Top.
  \]
\end{proposition}

\begin{proof}
  Since the plots of $\dg(Y) \in \Diff$ correspond to the
  plots of $Y \in \Diff$ of the form $\bas{U} \to Y$ for
  $U \in \bEucOp$, the identity of $\abs{\dg(Y)} = \abs{Y}$
  determines a continuous map $L(\dg(Y)) \to \bas{\!L}(Y)$,
  hence the natural transformation
  $L \circ \dg \to \bas{\!L}$.

  The natural transformation $L \to \bas{\!L} \circ \tr$ is
  given by the composition
  \[
    L(X) \xrightarrow{L\eta} L(\dg(\bas{\!X})) \to
    \bas{\!L}(\bas{\!X}) \quad (X \in \Diff)
  \]
  where $\eta \colon X \to \dg(\bas{\!X})$ is the unit of
  the adjunction
  $\tr \dashv \dg \colon \Diff \rightleftarrows \bDiff$.
\end{proof}

\begin{definition}
  \label{dfn:extended_smooth_map}
  By extending the notion of smoothness, we say that a set
  map $f$ from $X \in \Diff$ to $Y \in \bDiff$ is smooth if
  it is smooth as a map $X \to \dg(Y)$.  It follows by
  Proposition~\ref{prp:natural_extension_is_continuous} that
  such an $f$ is continuous as a map
  $L(X) \to \bas{\!L}(Y)$.
\end{definition}

% \medskip
\begin{example}
  \label{ex:projection_is_smooth}
  The projections $\F^{\N} \to \nF$ and
  $\rM(\F^{\N}) \to \rF$ are smooth for $\F = \R,\, \C$.  In
  fact, the former takes a plot
  $\sigma = (\sigma_n) \colon U \to \F^{\N}$ to the plot
  $U \subset \nst{U} \xrightarrow{\nst{\!\sigma}} \nF$ of
  $\dg(\nF)$, where
  $\nst{\!\sigma}([u_n]) = [\sigma_n(u_n)] \in \nF$ for
  $[u_n] \in \nst{U}$, and similarly for the latter.
\end{example}

% In particular, the map
% $\Cinf(X,Y) \to \bCinf(\bas{\!X},\bas{Y})$ which assigns to
% a smooth map $f \colon X \to Y$ its natural extension
% $\bas{\!f} \colon \bas{\!X} \to \bas{Y}$ is smooth for any
% $X,\, Y \in \Diff$.  Thus we have the following.

% \medskip
\begin{proposition}
  \label{prop:contunuous_extension_of_function_spaces}
  The natural map $\Cinf(X,Y) \to \bCinf(\bas{\!X},\bas{Y})$
  which takes a smooth map $f \colon X \to Y$ to
  $\tr(f) = \bas{\!f} \colon \bas{\!X} \to \bas{Y}$ is
  smooth, and hence continuous with respect to the
  $D$-topology on $\Cinf(X,Y)$ and the $\bD$-topology on
  $\bCinf(\bas{\!X},\bas{Y})$.
\end{proposition}

\begin{proof}
  It suffices to show that the composition
  $U \xrightarrow{\sigma} \Cinf(X,Y) \xrightarrow{\tr}
  \bCinf(\bas{\!X},\bas{Y})$ is a plot of
  $\dg(\bCinf(\bas{\!X},\bas{Y}))$ for any $U \in \EucOp$
  and $\sigma \in \Cinf(X,Y)(U)$.
  % or, equivalently, for any plot of the form
  % $\tau|U \colon U \subset \bas{U} \xrightarrow{\tau}
  % \bas{\!X}$ with $\tau \in \bas{\!X}(\bas{U})$ the
  % composition $U \to \bas{Y}$,
  % $u \mapsto [\sigma(u)(\tau_n(u))]$, is a plot of $\dg(Y)$.
  To see this, let us take a plot of the form
  $U \subset \bas{U} \xrightarrow{\tau} \bas{\!X} \in
  \dg(\bas{\!X})(U)$ with $\tau \in \bas{\!X}(\bas{U})$, and
  show that the composition $U \to \bas{Y}$,
  $u \mapsto \tr(\sigma(u))(\tau(u))$, is a plot of
  $\dg(\bas{Y})$.
  Since $\tau$ is a plot of $\bas{\!X}$, it is (locally) the
  composition of an internal smooth map
  $[\tau_n] \colon \bas{U} \to \bas{W}$ followed by the
  natural map
  $q_{\nu} \colon \bas{W} = \bas{W} \times \{\nu\} \subset
  \bas{W} \times X(W) \to \bas{\!X}$ for some
  $\nu \in X(W)$.
  But then, we have
  \[
    \tr(\sigma(u))(\tau(u)) =
    \bas{\sigma(u)}(q_{\nu}([\tau_n(u)])) = q_{\sigma(u)
      \circ \nu}([\tau_n(u)]) \in \bas{Y} \quad (u \in U)
  \]
  where $q_{\sigma(u) \circ \nu} \colon \bas{W} \to \bas{Y}$
  is the natural map corresponding to
  $\sigma(u) \circ \nu \in Y(W)$.  This means that the map
  $u \mapsto \tr(\sigma(u))(\tau(u))$ factors as a
  composition $U \subset \bas{U} \to \bas{Y}$, hence belongs
  to $\dg(\bas{Y})(U)$.
\end{proof}

%%%%%%%%%%%%%%%%%%%%%%%%%%%%%%%%%%%%%%%%%%%%%%%%%%%%%%%%%%%%
% Section 5
%%%%%%%%%%%%%%%%%%%%%%%%%%%%%%%%%%%%%%%%%%%%%%%%%%%%%%%%%%%%
\section{Schwartz distributions as quasi-asymptotic
  functions}
%%%%%%%%%%
\subsection{The space of Schwartz distributions}
Let $U$ be an open subset of a Euclidean space $\R^k$.  The
space of Schwartz distributions $\D'(U)$ is a continuous
dual of the vector space $\D(U)$ of test functions
$U \to \F$ equipped with a locally convex topology such that
a sequence $\{\varphi_n\}$ converges to $\varphi \in \D(U)$
if and only if
\begin{enumerate}[topsep=4pt]
\item[(i)] there exists a bounded subset $M$ such that
  $\supp \varphi_n \subset M$ for all $n \in \N$, and
\item[(ii)] $\{D^{\alpha}\varphi_n\}$ converges uniformly to
  $D^{\alpha}\varphi$ for any multi-index $\alpha \in \N^k$.
\end{enumerate}
% $\D'(U)$ is a LCTVS.
Any locally integrable function $f \colon U \to \F$ gives
rise to a distribution %
\[
  T_f \colon \varphi \mapsto % \bracket{T_f}{\varphi} =
  \dint_U f(x)\varphi(x)\,dx \quad (\,\varphi \in \D(U)\,).
\]
Thus, $\D'(U)$ contains the space $\mathcal{C}^r(U)$ of
differentiable functions of class $C^r$ ($r \leq \infty$).

For any $T \in \D'(U)$ and $\alpha \in \N^k$ we can define
its derivative $D^{\alpha}T$ by
\[
  \bracket{D^{\alpha}T}{\varphi} =
  (-1)^{\abs{\alpha}}\bracket{T}{D^{\alpha}\varphi} \quad
  (\,\varphi \in \D(U)\,).
\]
This extends the usual partial derivative of smooth
functions because
\[
  D^{\alpha}(T_f) = T_{D^{\alpha}f} \ \text{for all
    $f \in \tCinf(U)$ and $\alpha \in \N^k$.}
\]
holds for all $f \in \tCinf(U)$ and $\alpha \in \N^k$.

% \medskip
\begin{theorem}
  \label{thm:continuous_inclusion_of_Dist}
  There is an injection
  $\I_U \colon \D'(U) \to \rCinf(\ras{U},\rF)$ of
  differential vector spaces which is continuous with
  respect to the $\rD$-topology on $\rCinf(\ras{U},\rF)$ and
  extends the inclusion $\tCinf(U) \to \rCinf(\ras{U},\rF)$.
\end{theorem}
% \medskip

% Here, the term ``continuous map from $X \in \Top$ to
% $Y \in \Diff$'' means a continuous map from $X$ to the
% underlying topological space of $Y$ equipped with the
% $D$-topology.
%
% Note also that the topology of $\tCinf(U)$ given by the
% family of seminorms\,
% $\sup_{\abs{\alpha} \leq l}\norm{D^{\alpha}f}_K$ ($K$ is
% compact in $U$ and $l \geq 0$) is finer than the
% $D$-topology of $\Cinf(U,\F)$.

To prove the theorem we reinterpret the space of
distributions in terms of smooth functionals
% on smooth function spaces
by utilizing the notion of ``convenient vector space''
introduced by \cite[2.6.3]{Frohlicher-Kriegl}.
Every topological vector space $X$ admits a canonical
diffeology characterized as
% whose plots are those $\sigma \colon U \to X$
% such that $\phi \circ \sigma \colon U \to \F$ is smooth for
% any continuous linear functional $\phi \colon X \to \F$.
% In other words, canonical diffeology is
the coarsest diffeology with respect to which every
continuous linear functional $X \to \F$ is smooth (cf.\
\cite[Section 2]{Kock-Reyes}).
A convenient vector space is a topological vector space $X$
such that a linear functional $\phi \colon X \to \F$ is
continuous if and only if it is smooth with respect to the
canonical diffeology of $X$.
With respect to the strong topology, $\D(U)$ is a convenient
vector space (cf.\ \cite[Remark 3.5]{Frohlicher-Kriegl}),
hence its topological dual $\D'(U)$ has the same underlying
space as the smooth dual $D'(U)$ of the smooth vector space
$D(U)$ of test functions on $U$ equipped with canonical
diffeology (cf.\ \cite[Definition 4.1]{Giordano-Wu}).

The $D$-topology of $D(U)$ is finer than the locally convex
topology of $\D(U)$ by \cite[Corollary 4.12]{Giordano-Wu}.
Hence we have the following.

% \medskip
\begin{lemma}
  \label{lmm:locally_convex_topology_is_finer_than_D_topology}
  The identity map $\D'(U) \to D'(U)$ is continuous with
  respect to the $D$-topology on $D'(U)$.
\end{lemma}
% \medskip

Consequently, Theorem~\ref{thm:continuous_inclusion_of_Dist}
follows from its smooth version below.  (Cf.\
Proposition~\ref{prop:contunuous_extension_of_function_spaces}.)

% \medskip
\begin{theorem}
  \label{thm:smooth_inclusion_of_Dist}
  There is a smooth injection of differential vector spaces
  \[
    I_U \colon D'(U) \to \rCinf(\ras{U},\rF)
  \]
  which extends the inclusion
  $\Cinf(U,\F) \to \rCinf(\ras{U},\rF)$.
\end{theorem}
% \medskip

To construct $I_U$, we first introduce as in \citep{GKOS} a
linear embedding $E'(U) \to \rCinf(\ras{U},\rF)$ and extend
it to $D'(U)$ by employing sheaf-theoretic construction.
Here $E'(U)$ denotes the smooth dual of $\Cinf(U,\F)$, that
is, the subspace of $\Cinf(\Cinf(U,\F),\F)$ consisting of
$\F$-linear functionals on $\Cinf(U,\F)$.  We may regard
$E'(U)$ as a smooth version of the space of compactly
supported distributions.

% \medskip
\begin{lemma}
  \label{lmm:compactly_supported_distributions}
  For any $U \in \EucOp$ we have the following.
  \begin{enumerate}
  \item The linear map $E'(U) \to D'(U)$ induced by the
    inclusion $D(U) \to \Cinf(U,\F)$ is a smooth injection.
  % \item Any test function $\eta \in D(U)$ defines a smooth
  %   homomorphism $D'(U) \to E'(U)$ which takes $T$ to the
  %   product $\eta\, T$.
  \item The homomorphism $D'(U) \to E'(U)$ given by the
    multiplication $T \mapsto \eta\, T$ is smooth for any
    $\eta \in D(U)$.
  \end{enumerate}
\end{lemma}

\begin{proof}
  Following \citep{Giordano-Wu} denote by
  $D^{\mathrm{s}}(U)$ the same set as $D(U)$ but considered
  as a subspace of $\Cinf(U,\F)$.  For any $K \Subset U$ let
  \[
    D^{\mathrm{s}}_K(U) = \{ f \in D^{\mathrm{s}}(U) \mid
    \supp f \subset K \} \subset D^{\mathrm{s}}(U).
  \]
  Then we see by \cite[Theorem
  2.3]{Kock-Reyes} (see also \cite[Theorem 4.5 and Corollary
  4.11]{Giordano-Wu}) that the following hold in $\Diff$:
  \begin{equation}
    \label{eq:D_is_included_in_Ds}
    D(U) \subset D^{\mathrm{s}}(U),\ \
    D_K(U) = D^{\mathrm{s}}_K(U) \quad (V \in \EucOp)
  \end{equation}
  In fact, \cite[Theorem 2.3]{Kock-Reyes} says that a set
  map $\sigma \colon V \to D(U) \subset \Cinf(U,\F)$ is a
  plot of $D(U)$ if and only if its transpose
  $\sigma^{\vee}$ belongs to
  $\Cinf(V \times U,\F) = \Cinf(V,\Cinf(U,\F))$ and is
  locally of uniformly bounded support, implying
  \eqref{eq:D_is_included_in_Ds}.
  % $D(U)(V) \subset D^{\mathrm{s}}(U)(V)$ and
  % $D_K(U)(V) = D^{\mathrm{s}}_K(U)(V)$ hold for all
  % $V \in \EucOp$.
  %
  Thus the inclusion $D(U) \subset \Cinf(U,\F)$ is smooth,
  and hence induces a smooth linear map $E'(U) \to D'(U)$
  which is injective because $D(U)$ is dense in
  $\Cinf(U,\F)$.

  Similarly, the multiplication by $\eta$ defines a smooth
  homomorphism
  \[
    \Cinf(U,\F) \to D^{\mathrm{s}}_K(U) = D_K(U) \to D(U) \
    \ (K = \supp\eta)
  \]
  which in turn induces
  $D'(U) \to E'(U),\ T \mapsto \eta\,T$.
\end{proof}

% Note that, as in the case of functionals on
% $D(U)$, a linear functional on $\Cinf(U,\F)$ is smooth if
% and only if it is a continuous linear functional on
% $\tCinf(U)$.
% Let $E'(U)$ be the smooth dual of $\Cinf(U,\F)$;
% \[
%   E'(U) := L^{\infty}(\Cinf(U,\F),\F) \subset
%   \Cinf(\Cinf(U,\F),\F)
% \]
% %
% To construct $I_U$, we first introduce as in
% \cite[1.2.2]{GKOS} a linear embedding
% $E'(U) \to \rCinf(\ras{U},\rF)$ and extend it to $D'(U)$ by
% employing sheaf-theoretic construction.

%%%%%%%%%%
For any $T \in E'(U)$ and $\varphi \in \Cinf(\R^k,\F)$
denote by $T * \varphi$ the function on $U$ given by the
formula
\[
  (T * \varphi)(x) = \bracket{\widetilde{T}}{\varphi_x}
  \quad (x \in U)
\]
where $\widetilde{T} \in E'(\R^k)$ is the evident extension
of $T$ and $\varphi_x(t) = \varphi(x-t)$ for any
$t \in \R^k$.

% \medskip
\begin{lemma}
  \label{lmm:modified_convolution}
  We have the following:
  \begin{enumerate}
  \item The correspondence $(T,\varphi) \mapsto T * \varphi$
    induces a smooth map
    $E'(U) \times \Cinf(\R^k,\F) \to \Cinf(U,\F)$.
  \item We have
    $D^{\alpha}(T * \varphi) = D^{\alpha}T * \varphi$ for
    any $\alpha \in \N^k$.
  \end{enumerate}
\end{lemma}

\begin{proof} To prove (1), it suffices to see that the
  following operations are smooth.
  \begin{enumerate}[topsep=4pt]
  \item[(i)] $\Cinf(\R^k,\F) \times U \to \Cinf(\R^k,\F)$, \
    $(\varphi,x) \mapsto \varphi_x$, and
  \item[(ii)] $E'(U) \times \Cinf(\R^k,\F) \to \F$, \
    $(T,\varphi) \mapsto \bracket{\widetilde{T}}{\varphi}$.
  \end{enumerate}
  It is evident that (i) is smooth.  The smoothness of (ii)
  is a consequence of the facts that the inclusion
  $E'(U) \to E'(\R^k)$ is smooth and that $E'(\R^k)$ is a
  smooth dual of $\Cinf(\R^k,\F)$.
  That (2) holds is clear by the definition.
\end{proof}

Now, let $\varrho \in \mathcal{S}(\R^k)$ be a smooth
function that has rapidly decreasing partial derivatives and
satisfies
% \begin{enumerate}
% \item[(i)] $\dint \varrho(x)\,dx = 1$
% \item[(ii)]
%   $\dint x^{\alpha} \varrho(x)\,dx = 0\ \text{for} \ \alpha
%   \in \N^k\setminus \{0\}$.
% \end{enumerate}
\begin{equation}
  \label{eq:mollifier_function}
  \text{(i)}\ \int \varrho(x)\,dx = 1, \qquad %
  \text{(ii)}\ \dint x^{\alpha} \varrho(x)\,dx = 0\
  \text{for} \ \alpha \in \N^k\setminus \{0\}
\end{equation}
and define $j_U \colon E'(U) \to \Cinf(U^{\N},\F^{\N})$ by
\[
  j_U(T)((x_n)) = ((T * \varrho_n)(x_n)) \quad (\, T \in
  E'(U),\ (x_n) \in U^{\N}\,)
\]
where $\varrho_n(x) = n^k\varrho(nx)$ for $x \in \R^k$ and
$n \in \N$.  Observe that $\varrho_n \in \mathcal{S}(\R^k)$
enjoys the properties similar to
\eqref{eq:mollifier_function}.
Also, let $\rM(U^{\N}) = U^{\N} \cap \rM(\R^{\N})^k$, so
that $\ras{U}$ is the image of $\rM(U^{\N})$ under the
projection $\rM(\R^{\N})^k \to \rR^k$.

% \medskip
\begin{lemma}
  \label{lmm:inclusion_of_Dist_into_product}
  The following hold for every $T \in E'(U)$.
  \begin{enumerate}
  \item
    $j_U(D^{\alpha}T)((x_n)) = D^{\alpha}(j_U(T))((x_n))$
    for any $\alpha \in \N^k$ and $(x_n) \in U^{\N}$.
  \item $j_U(T)((x_n)) \in \rM(\F^{\N})$ for any
    $(x_n) \in \rM(U^{\N})$.
  \item If $(x_n),\, (y_n) \in \rM(U^{\N})$ and
    $(x_n) - (y_n) \in \rN(\R^{\N})^k$ then
    $j_U(T)((x_n)) - j_U(T)((y_n)) \in \rN(\F^{\N})$.
  \item If $j_U(T)((x_n)) \in \rN(\F^{\N})$ holds for all
    $(x_n) \in \rM(U^{\N})$ then $T = 0$.
  \item If $T = T_f$ for some $f \in D(U)$ then we have
    $j_U(T)((x_n)) - (f(x_n)) \in \rN(\F^{\N})$ for any
    $(x_n) \in U^{\N}$.
  \end{enumerate}
\end{lemma}

\begin{proof}
  (1)
  $j_U(D^{\alpha}T)((x_n)) = ((D^{\alpha}T *
  \varrho_n)(x_n)) = (D^{\alpha}(T * \varrho_n)(x_n)) =
  D^{\alpha}(j_U(T))((x_n))$ by
  Lemma~\ref{lmm:modified_convolution} (2).

  (2) Every $T \in E'(U)$ is of the form
  $\sum_{\abs{\alpha} \leq r} D^{\alpha}f_{\alpha}$, where
  $f_{\alpha}$ is a continuous function (regarded as a
  distribution) such that $\supp f_{\alpha}$ is contained in
  an arbitrary neighborhood of $\supp T$ (see
  \cite[1.2.9]{GKOS}).  Thus we need only consider the case
  $T = D^{\alpha}f$ with
  $f \in \mathcal{C}^0_c(U) \subset \tCinf(\R^k)$.  But
  then, we have
  \begin{align*}
    (T * \varrho_n)(x_n) = \int f(x_n-t)
    D^{\alpha}\varrho_n(t)\,dt
    &= \int f(x_n-t) n^{\abs{\alpha}+k}
      D^{\alpha}\varrho(nt)\,dt
    \\
    &= n^{\abs{\alpha}} \int f(x_n-s/n)
      D^{\alpha}\varrho(s)\,ds.
  \end{align*}
  Hence for each $n \geq 2$ there exists $c_n \in \N$ such
  that $\abs{(T * \varrho_n)(x_n)} \leq n^{c_n}$, implying
  $((T * \varrho_n)(x_n)) \in \rM(\F^{\N})$.
  % for all $(x_n) \in U^{\N}$.

  (3) By the local Lipshitz continuity of $T * \varrho_n$
  there exists $K_n > 0$ for each $n \in \N$ that satisfies
  \[
    \abs{(T * \varrho_n)(x_n) - (T * \varrho_n)(y_n)} \leq
    K_n\abs{x_n - y_n}.
  \]
  Here we may assume $(K_n) \in \rM(\R^{\N})$ because
  $(T * \varrho_n)$ takes values in $\rM(\F^{\N})$.  It
  follows, in particular, that if
  $(\abs{x_n-y_n}) \in \rN(\R^{\N})$ then
  $(\abs{(T * \varrho_n)(x_n) - (T * \varrho_n)(y_n)}) \in
  \rN(\F^{\N})$, and hence
  \[
    j_U(T)((x_n)) - j_U(T)((y_n)) = ((T * \varrho_n)(x_n) -
    (T * \varrho_n)(y_n)) \in \rN(\F^{\N}).
  \]

  (4) For any $\varphi \in \Cinf(U,\F)$ there holds
  $\bracket{T * \varrho_n}{\varphi} \to
  \bracket{T}{\varphi}$ since $T * \varrho_n \to T$ in
  $E'(U)$.  On the other hand, we have
  $\bracket{T * \varrho_n}{\varphi} \to 0$ because the
  condition $((T * \varrho_n)(x)) \in \rN(\F^{\N})$ implies
  $T * \varrho_n \to 0$ uniformly on $\supp T$.  Hence we
  have $\bracket{T}{\varphi} = 0$ for any
  $\varphi \in \Cinf(U,\F)$, implying $T = 0$.

  (5) Let $j_U(T)((x_n)) - (f(x_n)) = (\Delta_n(x_n))$,
  so that we have
  \[
    \Delta_n(x_n) = (T * \varrho_n)(x_n) - f(x_n) = \int
    (f(x_n-t) - f(x_n))\,\varrho_n(t)\,dt.
  \]
  By the Taylor expansion formula applied to $f$ we can
  write $f(x_n-t) - f(x_n) = P(d,t) + R(d,t)$, where
  \[
    % P(d,t) = \sum_{1 \leq \abs{\alpha} < d}
    % \frac{D^{\alpha}f(x_n)}{\alpha\,!}(-t)^{\alpha},\ \
    % R(d,t) = \sum_{\abs{\beta} = d}
    % \frac{D^{\beta}f(x_n-\theta t)}{\beta\,!}\,
    % (-t)^{\beta}\ \ (0 < \theta < 1).
    P(d,t) = \sum_{1 \leq \abs{\alpha} < d}
    \frac{D^{\alpha}f(x_n)}{\alpha\,!}(-t)^{\alpha}\ \
    \text{and} \ \
    R(d,t) = \sum_{\abs{\beta} = d}
    \frac{D^{\beta}f(x_n-\theta t)}{\beta\,!}\,
    (-t)^{\beta} % \ \ (0 < \theta < 1).
  \]
  with $\epsilon \in (0,1)$.
  But we have $\dint P(d,t)\,\varrho_n(t)\,dt = 0$ because
  $\dint t^{\alpha}\varrho_n(t)\,dt = 0$ for
  $\alpha \neq 0$.  Hence
  \begin{equation*}
    \begin{split}
      \Delta_n(x)
      % &= \int (f(x-t) - f(x))\,\varrho_n(t)\,dt
      % \\
      &= \int (P(d,t) + R(d,t))\,\varrho_n(t)\,dt
      \\
      &= \int R(d,t)\,n^k\varrho(nt)\,dt
      \\
      &= \int R(d,s/n)\,\varrho(s)\,ds
      \\
      &= \sum_{\abs{\beta} = d} \int
        \dfrac{D^{\beta}f(x_n-\theta s/n)}{\beta\,!}\,
        (-s/n)^{\beta}\,\varrho(s)\,ds
      \\
      &= n^{-d} \sum_{\abs{\beta} = d} \int
        \dfrac{D^{\beta}f(x_n-\theta s/n)}{\beta\,!}\,
        (-s)^{\beta}\,\varrho(s)\,ds.
    \end{split}
  \end{equation*}
  Since $f$ is compactly supported and
  $\varrho \in \mathcal{S}(\R^k)$, we can attain
  $\abs{\Delta_n(x_n)} \leq n^{-d_n}$ for any $d_n \in \N$
  ($n \geq 2$) by taking $d$ sufficiently large.
  Consequently, we have
  $j_U(f)((x_n)) - (f(x_n)) = (\Delta_n(x_n)) \in
  \rN(\F^{\N})$.
\end{proof}

As an immediate consequence of the lemma above we have the
following.

% \medskip
\begin{proposition}
  \label{prp:embedding_of_compactly_supported_distributions}
  There is an smooth injection of differential vector
  spaces\/
  \[
    J_U \colon E'(U) \to \rCinf(\ras{U},\rF)
  \]
  such that
  $J_U|D(U) = i_U|D(U) \colon D(U) \to \rCinf(\ras{U},\rF)$
  holds.
\end{proposition}

\begin{proof}
  By (2) and (3) of
  Lemma~\ref{lmm:inclusion_of_Dist_into_product}, we can
  assign to any $T \in E'(U)$ a well-defined function
  $J_U(T) \colon \ras{U} \to \rF$ given by the formula:
  \[
    J_U(T)([x_n]) = [j_U(T)(x_n)] \in \rF \quad (\,(x_n) \in
    \rM(U^{\N})\,).
  \]
  This formula together with the smoothness of $j_U(T)$
  implies $J_U(T) \in \rCinf(\ras{U}.\rF)$, and we obtain a
  map $J_U \colon E'(U) \to \rCinf(\ras{U},\rF)$.
  That $J_U$ is smooth is proved by arguing as in the proof
  of
  Proposition~\ref{prop:contunuous_extension_of_function_spaces}.
  Moreover, (1) and (4) of
  Lemma~\ref{lmm:inclusion_of_Dist_into_product} show that
  $J_U$ is an injection of differential vector spaces.
  Finally, Lemma~\ref{lmm:inclusion_of_Dist_into_product}
  (5) implies that $J_U$ coincides with $i_U$ on $D(U)$.
\end{proof}

In order to extend $J_U$ to an injection
$D'(U) \to \rCinf(\ras{U},\rF)$ we need the proposition
below.  Denote by $\mathcal{A}$ the sheaf of rings
$U \mapsto \Cinf(U,\F)$.

% \medskip
\begin{proposition}
  \label{prp:fineness_of_rCinf}
  The sheaf of $\mathcal{A}$-modules
  $U \mapsto \rCinf(\ras{U},\rF)$ is fine, i.e.\ admits
  a partition of unity.
\end{proposition}

\begin{proof}
  Let $\mathfrak{U} = \{U_{\lambda}\}$ be an open covering
  and $\{\chi_j\}_{j \in \N}$ be a smooth partition of unity
  subordinate to $\mathfrak{U}$.  Assign to each $j \in \N$
  a sheaf morphism
  $\chi_{j\,*} \colon \rCinf(\ras(-),\rF) \to
  \rCinf(\ras(-),\rF)$ defined by
  \[
    % \chi_{j\,*}(f) = (\tr(\chi_j)|\ras{V})f \quad (\,f \in
    % \rCinf(\ras{V},\rF)\,)
    \chi_{j\,*}(f) = \ras{\chi_j}f\,|\,\ras{U} \cap
    \ras{W}_j \quad (\,f \in \rCinf(\ras{U},\rF)\,)
  \]
  where $W_j \subset \R^k$ is an open subset satisfying
  $\supp\chi_j \subset W_j \subset \overline{W_j} \subset
  U_{\lambda(j)} \in \mathfrak{U}$.
  Then we can show that $\{\chi_{j\,*}\}$ provides a
  partition of unity, that is, satisfies
  $\sum_{j \in \N} \chi_{j\,*} = 1$ and
  $\supp \chi_{j\,*} \subset U_{\lambda(j)}$ for every
  $j \in \N$ as we have
  $\supp{(\ras{\chi_j})} = \ras{\supp{\chi_j}}$.
\end{proof}

Let $\mathfrak{U} = \{U_{\lambda}\}$ be an open covering of
$U$ such that $\overline{U}_{\!\lambda}$ is compact for
every $\lambda$ and $\{\psi_{\lambda}\}$ be a family of
elements of $D(U)$ such that $\psi_{\lambda} \equiv 1$ on a
neighborhood of\, $\overline{U}_{\!\lambda}$.
% Then we have the following.

% \medskip
\begin{lemma}
  \label{lmm:coherence_of_J_U}
  The following hold.
  \begin{enumerate}
  \item If $T \in E'(U)$ then
    $\supp J_U(T) = \ras{\supp T} \subset \ras{U}$.
  \item % We have
    $J_U(\psi_{\lambda}T)|\ras{U}_{\lambda} \cap
    \ras{U}_{\mu} = J_U(\psi_{\mu}T)|\ras{U}_{\lambda} \cap
    \ras{U}_{\mu}$ for any $T \in D'(U)$ and
    $\lambda,\,\mu \in \Lambda$.
  \end{enumerate}
\end{lemma}

\begin{proof}
  (1) To prove the inclusion
  $\supp J_U(T) \subset \ras{\supp T}$, it suffices to show
  that $x \in \ras{U} \setminus \ras{\supp T}$ implies
  $J_U(T)(x) = 0$.
  Suppose $x$ is represented by a net
  $(x_n) \in \rM(U^{\N})$ such that $x_n \not\in \supp T$
  for every $n \in \N$.  Then, by arguing as in
  \citep[Proposition 1.2.12]{GKOS} we can show that
  $\abs{(T * \varrho_n)(x_n)} \leq 1/n^{d_n}$ holds for any
  $d_n \in \N$.  Thus we have
  $j_U(T)((x_n)) \in \rN(\F^{\N})$, implying
  $J_U(T)(x) = 0$.  The converse inclusion
  $\supp J_U(T) \supset \ras{\supp T}$ is proved again by
  using the argument of \citep[ Proposition 1.2.12]{GKOS}.

  (2) This is clear from the fact that
  $\psi_{\lambda} \equiv \psi_{\lambda} \equiv 1$ on
  $U_{\lambda} \cap U_{\mu}$.
\end{proof}

Proposition~\ref{prp:fineness_of_rCinf} and
Lemma~\ref{lmm:coherence_of_J_U} enable us to extend
$J_U \colon E'(U) \to \rCinf(\ras{U},\rF)$ to a linear map
$I_U \colon D'(U) \to \rCinf(\ras{U},\rF)$ by putting
\[
  I_U(T) = \tsum_{j \in \N}\,
  \chi_{j\,*}(J_U(\psi_{\lambda(j)}\,T)) \quad %
  (\,T \in D'(U)\,)
\]
where $\{\chi_j\}$ is a smooth partition of unity
subordinate to the open cover $\mathfrak{U}$ of $U$.

Denote by $\rCinf(\ras{U},\rF)_0$ the differential
subalgebra of $\rCinf(\ras{U},\rF)$ consisting of those
$f \colon \ras{U} \to \rF$ represented by a net of smooth
functions $(f_n) \in \Cinf(U,\F)^{\N}$, so that we have
$f(x) = [f_n(x_n)]$ for every $x = [x_n] \in \ras{U}$.  By
the construction, $I_U$ factors as a composition
\[
  D'(U) \to \rCinf(\ras{U},\rF)_0 \xrightarrow{\subset}
  \rCinf(\ras{U},\rF)
\]
and there is a bilinear pairing
$\rCinf(\ras{U},\rF)_0 \times D(U) \to \rF$ which takes
$([f_n],\varphi)$ to $[\bracket{f_n}{\varphi}]$, where
\[
  \bracket{f_n}{\varphi} = \int_U f_n(x)\varphi(x)\,dx \quad
  (n \in \N).
\]
We show that this pairing is compatible with the duality
pairing $D'(U) \times D(U) \to \F$ under $I_U$.

% \medskip
\begin{lemma}
  \label{lmm:preservation of pairings}
  We have $\bracket{I_U(T)}{\varphi} =
  \bracket{T}{\varphi}$ for every $T \in D'(U)$ and $\varphi
  \in D(U)$.
\end{lemma}

\begin{proof}
  Let $\varphi \in D(U)$ and
  $M = \max\{j \in \N \mid \supp \varphi \cap \supp \chi_j
  \neq \emptyset\}$.  Then we have
  $\sum_{j \leq M} \chi_j = 1$ on $\supp \varphi$, and hence
  \[
    \bracket{I_U(T)}{\varphi} = \underset{j \in \N}{\tsum}\,
    \bracket{\chi_{j\,*}(J_U(\psi_{\lambda(j)}\,T))}{\varphi}
    = [\bracket{T}{\underset{j \leq M}{\tsum}\,
      \chi_j\psi_{\lambda(j)}\,\varphi}] =
    [\bracket{T}{\varphi}] = \bracket{T}{\varphi}
  \]
  holds for every $T \in D'(U)$.
\end{proof}

To complete the proof of
Theorem~\ref{thm:smooth_inclusion_of_Dist} it suffices to
verify the following.

% \medskip
\begin{proposition}
  \label{prp:properties_of_I_U}
  Let $U$ be a open subset of $\R^k$.  Then we have the
  following.
  \begin{enumerate}
  \item $I_U \colon D'(U) \to \rCinf(\ras{U},\rF)$ does not
    depend on the choice of $\{U_{\lambda}\}$,
    $\{\psi_{\lambda}\}$ and $\{\chi_j\}$.
  \item $I_U \colon D'(U) \to \rCinf(\ras{U},\rF)$ is a
    smooth injection of differential vector spaces.
  \item
    $I_U|E'(U) = J_U \colon E'(U) \to \rCinf(\ras{U},\rF)$.
  \item
    $I_U|\Cinf(U,\F) = i_U \colon \Cinf(U,\F) \to
    \rCinf(\ras{U},\rF)$.
  \end{enumerate}
\end{proposition}

\begin{proof}
  (1) This follows from the fineness of the sheaf
  $U \mapsto \rCinf(\ras{U},\rF)$ and Lemma
  \ref{lmm:coherence_of_J_U}.  (See \citep[Proposition
  1.2.18]{GKOS}.)

  (2) It is evident that $I_U$ is a linear map of smooth
  vector spaces.  Its injectivity follows from
  Lemma~\ref{lmm:preservation of pairings} because
  $I_U(T) = 0$ implies that $\bracket{T}{\varphi} = 0$ holds
  for every $\varphi \in D(U)$, meaning $T = 0$.  That $I_U$
  commutes with differential operators can be proved as in
  \citep[Proposition 1.2.17]{GKOS} by using the fact that
  $\psi_{\lambda(j)} = 1$ on a neighborhood of the closure
  of $U_{\lambda(j)}$.

  (3) Let $T \in E'(U)$.  To verify $I_U(T) = J_U(T)$, it
  suffices to show that $J_U(T) - I_U(T) = 0$ holds on each
  $\ras{U}_{\lambda}$.  Let
  $F = (1-\psi_{\lambda})T \in E'(U)$.  Then for any
  $x = [x_n] \in \ras{U}_{\lambda}$ we have
  \[
    J_U(T)(x) - I_U(T)(x) = [(F * \varrho_n)(x_n)] =
    J_U(F)(x) = 0
  \]
  % because
  % $\ras{U}_{\lambda} \subset (\ras{\supp F})^{\,c} = \supp
  % J_U(F)^{\,c}$ holds by the definition of $F$ and
  % Lemma~\ref{lmm:coherence_of_J_U} (1).
  because
  $\supp J_U(F) = \ras{\supp F} \subset \ras{(U \setminus
    \overline{U}_{\lambda})}$ by the definition of $F$ and
  Lemma~\ref{lmm:coherence_of_J_U} (1).

  (4) Let $f \in \Cinf(U,\F)$.  Then for any
  $\lambda \in \Lambda$ and
  $x = [x_n] \in \ras{U}_{\lambda}$ we have
  \[
    I_U(f)(x) = I_U(\psi_{\lambda}f)(x) =
    J_U(\psi_{\lambda}f)(x) = i_U(\psi_{\lambda}f)(x) =
    i_U(f)(x)
  \]
  by (3) above and
  Proposition~\ref{prp:embedding_of_compactly_supported_distributions}.
  Thus $I_U|\Cinf(U,\F) = i_U$ holds.
\end{proof}

\begin{remark}
  Let $\EucOp_0$ denote the sub-site of $\EucOp$ having the
  same objects as $\EucOp$ but only smooth injections as its
  morphisms.  Then both $U \mapsto D'(U)$ and
  $U \mapsto \rCinf(\ras{U},\rF)$ are sheaves on $\EucOp_0$
  and the injections
  $I_U \colon D'(U) \to \rCinf(\ras{U},\rF)$ determine a
  sheaf morphism that restricts to the identity on the
  common subsheaf $U \mapsto \Cinf(U,\F)$.  (Compare
  \cite[1.2.20]{GKOS}.)
\end{remark}

%%%%%%%%%%%%%%%%%%%%%%%%%%%%%%%%%%%%%%%%%%%%%%%%%%%%%%%%%%%%
% Section 6
%%%%%%%%%%%%%%%%%%%%%%%%%%%%%%%%%%%%%%%%%%%%%%%%%%%%%%%%%%%%
\section{The relation with Colombeau's algebra}
We first recall briefly the definition of the special
Colombeau algebra of generalized functions and its smooth
version introduced by \citep{Giordano-Wu}.
% Henceforth, we denote $\Cinf(U) = \Cinf(U,\F)$ and
% similarly for $\tCinf$ unless it is necessary to specify
% $\F$.

\subsection{Special Colombeau algebra}
Let $J = (0,1]$ and consider the space
$\tCinf(U)^{J} = \{ (f_{\epsilon}) \mid \epsilon \in J \}$
of $J$-nets in the locally convex space of $\F$-valued
smooth functions on an open subset $U$ of $\R^k$.  Then the
special Colombeau algebra on $U$ is defined to be the
quotient
\[
  \Gs(U) = \sM(\tCinf(U)^J)/\sN(\tCinf(U)^J)
\]
of the subalgebra of ``moderate nets'' over that of
``negligible nets.''  More explicitly,
\begin{align*}
  % \sM(\tCinf(U)^J)
  % &= \{ (f_{\epsilon}) \in \tCinf(U)^J \mid %
  %   \forall \alpha \in \N^k,\ \forall K \Subset U,\ %
  %   \exists c \in \N,\ \norm{D^{\alpha}f_{\epsilon}(x)}_K =
  %   O(\epsilon^{-c}) \},
  % \\
  % \sN(\tCinf(U)^J)
  % &= \{ (f_{\epsilon}) \in \tCinf(U)^J \mid %
  %   \forall \alpha \in \N^k,\ \forall K \Subset U,\ %
  %   \forall d \in \N,\ \norm{D^{\alpha}f_{\epsilon}(x)}_K =
  %   O(\epsilon^d) \}
  \sM(\tCinf(U)^J)
  &= \{ (f_{\epsilon}) \mid %
    \forall \alpha \in \N^k,\ \forall K \Subset U,\ %
    \exists c \in \N,\ \norm{D^{\alpha}f_{\epsilon}(x)}_K =
    O(\epsilon^{-c}) \},
  \\
  \sN(\tCinf(U)^J)
  &= \{ (f_{\epsilon}) \mid %
    \forall \alpha \in \N^k,\ \forall K \Subset U,\ %
    \forall d \in \N,\ \norm{D^{\alpha}f_{\epsilon}(x)}_K =
    O(\epsilon^d) \}
\end{align*}
where $K \Subset U$ means that $K$ is compact in $U$ and
$\norm{D^{\alpha}f_{\epsilon}(x)}_K := \max_{x \in K}
\abs{D^{\alpha}f_{\epsilon}(x)}$.

As described in \cite[Theorem 1.1]{Giordano-Wu},
the correspondence
$U \mapsto \Gs(U)$ is a fine and supple sheaf
% of differential algebras defined
on $\EucOp_0$ (cf.\ Remark at the end of Section~5), and
there is a sheaf embedding of differential vector spaces
$\iota_U \colon \D'(U) \to \Gs(U)$ that restricts to the
inclusion $\sigma \colon \tCinf(U) \to \Gs(U)$ induced by
the diagonal inclusion $\tCinf(U) \to \tCinf(U)^J$.  Beware
that $\iota_U$ is not uniquely determined as it depends on
the choice of mollifier.

%%%%%%%%%%
\subsection{Colombeau algebra as a diffeological space}
\label{subsec:smooth Colombeau algebra}
Giordano and Wu introduced in \cite[\S{5}]{Giordano-Wu} a
smooth version of $\Gs(U)$ as a quotient
\[
  G^s(U) := \cM(\Cinf(U)^J)/\cN(\Cinf(U)^J)
\]
where $\cM(\Cinf(U)^J)$ and $\cN(\Cinf(U)^J)$ have the same
underlying sets as $\sM(\Cinf(U)^J)$ and $\sN(\Cinf(U)^J)$,
respectively, but are regarded as subspaces of the
diffeological product space $\Cinf(U)^J = \Cinf(U,\F)^J$.

% \medskip
\begin{theorem}[{\cite[Theorem 5.1]{Giordano-Wu}}]
  \label{thm:smooth Colombeau algebra}
  The space $G^s(U)$ is a smooth differential algebra that
  admits a linear embedding $i_U \colon D'(U) \to G^s(U)$
  which {\rm (i)}\! is smooth, {\rm (ii)}\! preserves
  partial derivatives, and {\rm (iii)}\! restricts to the
  inclusion $\sigma \colon \Cinf(U) \to G^s(U)$ induced by
  the diagonal inclusion $\Cinf(U) \to \Cinf(U)^J$.
\end{theorem}
% \medskip

The theorem below describes the relationship between
$G^s(U)$ and $\rCinf(\ras{U},\rF)$.  Denote by $\aF$ the
non-Archimedean field of asymptotic numbers (cf.\ Remark at
the end of Sections~2).

% \medskip
\begin{theorem}
  \label{thm:relation_between_G_and_rCinf}
  There is a diagram of sheaves on $\EucOp_0$
  \begin{equation}
    \label{eq:Colombeau_and_nonstandard}
    \vcenter{
      \xymatrix{%
        D'(U) \ar[r]^-{i_U} \ar[d]_-{I_U}
        \ar[rd]^-{i'_U} %
        & G^s(U) \ar[r]^-{k_U} %
        & F^a(U) %
        \\
        \rCinf(\ras{U},\rF) %
        & F^{\rho}(U) \ar[l]_-{j_U} %
        & F^a_{\rho}(U) \ar@{->>}[u]%_-{p_U}
        \ar[l]_-{\supset} %
      }  %
    }
    \qquad (U \in \EucOp_0)
  \end{equation}
  enjoying the following properties:
  \begin{enumerate}
  \item $F^a(U)$ is a smooth differential algebra over $\aF$
    that contains $\Cinf(U)$ as a differential subalgebra.
  \item $F^{\rho}(U)$ is a smooth differential algebra over
    $\rF$ that contains $\Cinf(U)$ as a differential
    subalgebra.
  \item There exists a smooth surjection from a differential
    subalgebra $F^a_{\rho}(U)$ of $F^{\rho}(U)$ onto
    $F^a(U)$.
  \item $k_U$ and $j_U$ are smooth homomorphisms of
    differential algebras restricting to the identity on
    $\Cinf(U)$.
  \item $i'_U$ is a smooth injection of differential vector
    spaces such that $I_U = j_U \circ i'_U$.
  \end{enumerate}
\end{theorem}

\begin{proof}
  We put
  \[
    F^a(U) = \aM(\Cinf(U)^{\N})/\aN(\Cinf(U)^{\N}), \quad
    F^{\rho}(U) = \rM(\Cinf(U)^{\N})/\rN(\Cinf(U)^{\N})
  \]
  for each open subset $U \subset \R^k$, where
  \begin{align*}
    % \aM(\Cinf(U)^{\N})
    % &= \{ (f_n) \in \Cinf(U)^{\N} \mid \forall \alpha \in
    %   \N^k,\ \forall K \Subset U,\ \exists c \in \N,\
    %   \norm{D^{\alpha}f_n}_{K} \leq n^{c}\ \aev\, \}
    % \\
    % \aN(\Cinf(U)^{\N})
    % &= \{ (f_n) \in \Cinf(U)^{\N} \mid \forall \alpha \in
    %   \N^k,\ \forall K \Subset U,\ \forall d \in \N,\
    %   \norm{D^{\alpha}f_n}_{K} \leq n^{-d}\ \aev\, \}
    % \\
    % \rM(\Cinf(U)^{\N})
    % &= \{ (f_n) \in \Cinf(U)^{\N} \mid \forall \alpha \in
    %   \N^k,\ \forall (K_n) \Subset U^{\N},\ \exists
    %   (c_n) \in \N^{\N},\ \norm{D^{\alpha}f_n}_{K_n} \leq
    %   n^{c_n}\ \aev\, \}
    % \\
    % \rN(\Cinf(U)^{\N})
    % &= \{ (f_n) \in \Cinf(U)^{\N} \mid \forall \alpha \in
    %   \N^k,\ \forall (K_n) \Subset U^{\N},\ \forall
    %   (d_n) \in \N^{\N},\ \norm{D^{\alpha}f_n}_{K_n} \leq
    %   n^{-d_n}\ \aev\, \}.
    \aM(\Cinf(U)^{\N})
    &= \{ (f_n) \mid \forall \alpha \in
      \N^k,\ \forall K \Subset U,\ \exists c \in \N,\
      \norm{D^{\alpha}f_n}_{K} \leq n^{c}\ \aev\, \}
    \\
    \aN(\Cinf(U)^{\N})
    &= \{ (f_n) \mid \forall \alpha \in
      \N^k,\ \forall K \Subset U,\ \forall d \in \N,\
      \norm{D^{\alpha}f_n}_{K} \leq n^{-d}\ \aev\, \}
    \\
    \rM(\Cinf(U)^{\N})
    &= \{ (f_n) \mid \forall \alpha \in
      \N^k,\ \forall (K_n) \Subset U^{\N},\ \exists
      (c_n) \in \N^{\N},\ \norm{D^{\alpha}f_n}_{K_n} \leq
      n^{c_n}\ \aev\, \}
    \\
    \rN(\Cinf(U)^{\N})
    &= \{ (f_n) \mid \forall \alpha \in
      \N^k,\ \forall (K_n) \Subset U^{\N},\ \forall
      (d_n) \in \N^{\N},\ \norm{D^{\alpha}f_n}_{K_n} \leq
      n^{-d_n}\ \aev\, \}.
  \end{align*}
  It is clear by the definition that $F^a(U)$ and
  $F^{\rho}(U)$ are smooth differential algebras over $\aF$
  and $\rF$, respectively.
  Moreover, as there is a sequence of inclusions
  \[
    \rN(\Cinf(U)^{\N}) \subset \aN(\Cinf(U)^{\N}) \subset
    \aM(\Cinf(U)^{\N}) \subset \rM(\Cinf(U)^{\N}),
  \]
  we see that $F^a(U)$ is a subquotient of $F^{\rho}(U)$ of
  the form $F^a_{\rho}(U)/K^a_{\rho}(U)$, where
  \[
    F^a_{\rho}(U) = \aM(\Cinf(U)^{\N})/\rN(\Cinf(U)^{\N}),
    \quad K^a_{\rho}(U) =
    \aN(\Cinf(U)^{\N})/\rN(\Cinf(U)^{\N}).
  \]
  % proving (3).

  The homomorphism $\Cinf(U)^J \to \Cinf(U)^{\N}$ which
  takes $(f_{\epsilon})_{\epsilon \in J}$ to
  $(f_{1/n})_{n \in \N}$ restricts to
  $\cL(\Cinf(U)^J) \to \aL(\Cinf(U)^{\N})$ for both
  $\mathbf{L} = \mathbf{M},\, \mathbf{N}$ because
  ``$= O(\epsilon^p)$'' implies ``$\leq n^{-p-1}$ a.e.''
  under the correspondence $\N \ni n \mapsto 1/n \in J$.
  Thus we obtain a homomorphisms of SDAs
  $k_U \colon G^s(U) \to F^a(U)$.
  % The homomorphism $\Cinf(U)^J \to \Cinf(U)^{\N}$ which
  % takes $(f_{\epsilon})_{\epsilon}$ to $(f_{1/n})_n$ takes
  % $\cL(\Cinf(U)^J)$ into $\aL(\Cinf(U)^{\N})$ for
  % $\mathbf{L} = \mathbf{M},\, \mathbf{N}$ because
  % ``$= O(\epsilon^p)$'' implies ``$\leq n^{-p-1}$ a.e.''
  % under the correspondence $\N \ni n \mapsto 1/n \in J$.
  % Thus we obtain a homomorphisms of SDAs
  % $k_U \colon G^s(U) \to F^a(U)$.
  %
  Notice that $k_U$ is not injective unless $U = \emptyset$,
  because $G^s(\R^0)$ is a ring with zero-divisors, hence
  $G^s(\R^0) \to F^a(\R^0) = \aF$ cannot be injective.
  Still, the composition
  $k_U \circ i_U \colon D(U) \to F^a(U)$ is an injection
  (cf.\ Remark below).

  On the other hand,
  $j_U \colon F^{\rho}(U) \to \rCinf(\ras{U},\rF)$ is
  induced by the homomorphism
  $\Cinf(U)^{\N} \to \Cinf(\ras{U},\nF)$ which takes
  $(f_n) \in \Cinf(U)^{\N}$ to the function
  $\ras{U} \to \rF,\ [x_n] \mapsto [f_n(x_n)]$.  It is clear
  by the definition that
  $I_U \colon D'(U) \to \rCinf(\ras{U},\rF)$ factors as a
  composition of $j_U$ with a smooth injection of
  differential vector spaces
  $i'_U \colon D'(U) \to F^{\rho}(U)$.
\end{proof}

\begin{remark}
  The algebra $F^a(U)$ can be regarded a smooth version of
  the algebra $\ras{\mathbb{E}}(U)$ of asymptotic functions
  introduced by \citep{Oberguggenberger-Todorov}.  Although
  $\ras{\mathbb{E}}(U)$ does not contain $\Gs(U)$, there is
  an embedding of $\D'(U)$ into $\ras{\mathbb{E}}(U)$ (see
  Theorem 5.7 of \citep{Oberguggenberger-Todorov}).
\end{remark}

%%%%%%%%%%%%%%%%%%%%%%%%%%%%%%%%%%%%%%%%%%%%%%%%%%%%%%%%%%%%
% Section 7
%%%%%%%%%%%%%%%%%%%%%%%%%%%%%%%%%%%%%%%%%%%%%%%%%%%%%%%%%%%%
\section{Applications of asymptotic functions to homotopy
  theory of diffeological spaces}
Quasi-asymptotic maps can be expected to provide ``weak
solutions'' to problems that are hard or impossible to solve
within the usual framework of smooth maps.  In this section,
we present some modest applications related to homotopy
theory of diffeological spaces.

\subsection{Quasi-asymptotic maps between diffeological
  spaces}

\begin{definition}
  Given $X,\, Y \in \Diff$ denote by $\wCinf(X,Y)$ the
  diffeological subspace of $\dg(\rCinf(\ras{\!X},\ras{Y}))$
  consisting of such morphisms $\ras{\!X} \to \ras{Y}$ in
  $\rDiff$ that restrict to a set map $X \to Y$.  Members of
  $\wCinf(X,Y)$ are called quasi-asymptotic maps from $X$ to
  $Y$.
  There are a smooth inclusion $\Cinf(X,Y) \to\wCinf(X,Y)$
  given by the unit $1 \to \dg \circ \tr$ and a smooth
  composition
  \[
    \wCinf(Y,Z) \times \wCinf(X,Y) \to \wCinf(X,Z)
  \]
  induced by the composition in $\rDiff$.  Thus we obtain
  a category $\wDiff$ enriched over $\Diff$ with
  diffeological spaces as objects and $\wCinf(X,Y)$ as the
  set of morphism from $X$ to $Y$.
\end{definition}
% \medskip

Beware that $\wDiff$ is no longer a concrete category
because $\wCinf(X,Y)$ is not in general a subset of
$\hom_{\Set}(X,Y)$.

% \medskip
\begin{proposition}
  \label{prp:properties_of_wDiff}
  $\wDiff$ is closed under small limits and colimits.
  % and is cartesian closed with $\wCinf(X,Y)$ as exponential
  % objects.  %%% False!
\end{proposition}

\begin{proof}
  Since the forgetful functor $\rDiff \to \Set$ reflects
  small limits and colimits, so is its restriction
  $\wDiff \to \Set$.
\end{proof}

Typical examples of quasi-asymptotic maps are given by
piecewise smooth maps defined as follows.

% \medskip
\begin{definition}
  A continuous map $f \colon U \to \R^l$ defined on an open
  subset $U$ of $\R^k$ is said to be {piecewise smooth}
  if for each $x \in U$ there is an open subset
  $W \subset U$ such that $x \in \overline{W}$ and $f|W$ is
  smooth.
\end{definition}
% \medskip

The functions $\R^k \to \R$ below are apparently piecewise
smooth. 
\begin{quote}
  \begin{tabular}{ll}
    Euclidean norm: & $x \mapsto \norm{x}$ \\
    Maximum value:  & $(x_1,\cdots,x_k)
                      \mapsto \max\{x_1,\cdots,x_k\}$ \\
    Minimum value:  & $(x_1,\cdots,x_k)
                      \mapsto \min\{x_1,\cdots,x_k\}$
  \end{tabular}
\end{quote}
Note also that if $f \colon U \to \R^l$ is piecewise smooth
then so is its restriction % $f|V \colon V \to \R^l$
to arbitrary open subset $V \subset U$.

% \medskip
\begin{proposition}
  \label{prp:piecewise_smooth_map_is_quasi_asymptotic}
  Let $f$ be a continuous map from $X \subset \R^k$ to
  $Y \subset \R^l$.  Suppose $f$ extends to a piecewise
  smooth map $g \colon U \to \R^l$ defined on an open subset
  $U$ containing $X$.  Then $f$ is quasi-asymptotic, that
  is, there is $F \in \rCinf(\ras{X},\ras{Y})$ such that
  $F|X$ coincides with $f \colon X \to Y$.
\end{proposition}

\begin{proof}
  Let $g_j$ % ($1 \leq j \leq l$)
  be the $j$-th component of $g$ regarded as a distribution
  and put
  \[
    I_U(g) = (I_U(g_1),\cdots,I_U(g_l)) \in
    \rCinf(\ras{U},\rR^l).
  \]
  Let $\ras{\!g} \colon \ras{U} \to \rR^l$ be the continuous
  map which takes $[x_n] \in \ras{U}$ to
  $[g(x_n)] \in \rR^l$.  To prove the proposition it
  suffices to show that
  \begin{equation}
    \label{eq:piecewise_coincidence}
    I_U(g)(x) = \ras{\!g}(x)
    % = g(x) \ \ \text{for all} \ \ x \in U \subset \ras{U}
  \end{equation}
  holds for all $x \in U \subset \ras{U}$, because this
  means $F = I_U(g)|\,\ras{\!X}$ satisfies $F|X = f$.
  If $W$ is an open subset of $U$ such that $g$ is smooth on
  $W$ then \eqref{eq:piecewise_coincidence} holds for all
  $x \in \ras{W}$ by Proposition~\ref{prp:properties_of_I_U}
  (4).  But, as both $I_U(g)$ and $\ras{\!g}$ are continuous
  and $\rR^l$ is Hausdorff under $\bD$-topology,
  \eqref{eq:piecewise_coincidence} is valid for all
  $x \in \overline{\ras{W}} \subset \ras{U}$.
  Consequently, \eqref{eq:piecewise_coincidence} holds for
  all $x \in U$ because for every $x \in U$ there is an open
  subset $W \subset U$ such that
  $x \in \overline{W} \subset \overline{\ras{W}}$ and $g|W$
  is smooth.
\end{proof}

The following is also useful for constructing
quasi-asymptotic maps.

% \medskip
\begin{proposition}
  \label{prp:subduction_and_quasi_asymptoticity}
  Let $p \colon X \to Y$ and $f \colon Y \to Z$
  be maps between diffeological spaces.  Suppose $p$ is a
  smooth subduction.  Then $f$ is quasi-asymptotic if and
  only if so is the composite $f \circ p \colon X \to Z$.
\end{proposition}

\begin{proof}
  The ``only if'' part is obvious.  To prove the ``if''
  part, suppose that there exists
  $G \in \rCinf(\ras{\!X},\ras{\!Z})$ such that
  $G|X = f \circ p$ holds.
  Let $\ras{p} = \tr(f) \in \rCinf(\ras{\!X},\ras{\!Y})$.
  Then $\ras{p}$ is a subduction in $\rDiff$ because $\tr$
  preserves colimits.  Hence
  $G \colon \ras{\!X} \to \ras{\!Z}$ factors as a
  composition
  \[
    \ras{\!X} \xrightarrow{\ras{p}} \ras{Y} \xrightarrow{F}
    \ras{\!Z}
    % \quad (F \in \rCinf(\ras{Y},\ras{\!Z}))
  \]
  with $F \in \rCinf(\ras{Y},\ras{\!Z})$.  Thus we have
  $G|X = F|Y \circ p = f \circ p$, implying that $F|Y = f$
  holds as desired.
\end{proof}

%%%%%%%%%%
\subsection{Concatenation of quasi-asymptotic paths}
Let $I = [0,1]$ be the unit interval in $\R$ and $X$ be a
subspace of $\R^k$.  By an {quasi-asymptotic path} in
$X$ we mean a quasi-asymptotic map from $I$ to $X$.  Denote
by $\Path(X)$ the space of all quasi-asymptotic paths in
$X$, i.e.\ $\Path(X) = \wCinf(I,X)$.
We show that any two quasi-asymptotic paths
$\alpha,\, \beta$ such that $\alpha(1) = \beta(0)$ can be
concatenated together to form a new quasi-asymptotic path.
Let
\[
  \Path(X) \times_X \Path(X) = \{ (\alpha,\beta) \in
  \Path(X) \times \Path(X) \mid \alpha(1) = \beta(0) \}
\]
% The next proposition shows that {concatenation} of
% asymptotic paths can be realized as a smooth map from
% $\Path(X) \times_X \Path(X)$ to $\Path(X)$.

\begin{theorem}
  \label{thm:concatenation_of_asymptotic_paths}
  There is a smooth map
  $\Path(X) \times_X \Path(X) \xrightarrow{*} \Path(X)$
  % such that for any $(\alpha,\beta) \in \Path(X) \times_X
  % \Path(X)$ and $t \in I$ we have
  % \[
  %   \alpha \ast \beta\,(t) =
  %   \begin{cases}
  %     \alpha(2t), & 0 \leq t \leq 1/2
  %     \\
  %     \beta(2t-1), & 1/2 \leq t \leq 1
  %   \end{cases}
  % \]
  which takes
  $(\alpha,\beta) \in \Path(X) \times_X \Path(X)$ to a path
  $\alpha * \beta \in \Path(X)$ such that
  $\alpha \ast \beta\,(t)$ has value $\alpha(2t)$ if
  $0 \leq t \leq 1/2$ and $\beta(2t-1)$ if
  $1/2 \leq t \leq 1$.
\end{theorem}

\begin{proof}
  Let $\ell_1$ and $\ell_2$ be non-decreasing piecewise
  linear functions $\R \to I$ such that $\ell_1(t) = 2t$ for
  $0 \leq t \leq 1/2$ and $\ell_2(t) = 2t-1$ for
  $1/2 \leq t \leq 1$.
  Then $\ell_1|I$ and $\ell_2|I$ are quasi-asymptotic by
  Proposition~\ref{prp:piecewise_smooth_map_is_quasi_asymptotic},
  and hence so is $\alpha * \beta \colon I \to I$ given by
  $(\alpha * \beta)(t) = \alpha(\ell_1(t)) +
  \beta(\ell_2(t)) - \beta(0)$.  % \in X \subset \R^k$
\end{proof}

%%%%%%%%%%
\subsection{Smooth cell complexes and homotopy extension
  property}
As in the case of $\Top$ or
$\Diff$, there is a notion of homotopy in the category
$\bDiff$: a {homotopy} between morphisms $f,\, g \in
\bCinf(X,Y)$ is a morphism $H \in \bCinf(X \times
\bas{\!I},Y)$ such that $H \circ i_0 = f$ and $H \circ i_1 =
g$, where $i_{\alpha}$ is the inclusion $X \to X \times
\{\alpha\} \subset X \times \bas{\!I}$ for $\alpha = 0,\,1$.
Likewise, a {quasi-asymptotic homotopy} between
quasi-asymptotic maps $f,\,g \in
\wCinf(X,Y)$ is defined to be a quasi-asymptotic map $H \in
\wCinf(X \times I,Y)$ satisfying $H \circ i_0 = f$ and $H
\circ i_1 = g$.

The following is a diffeological analog to the notion of
relative cell complex due to \citep{Hovey} which plays a
crucial role in the homotopy theory of topological spaces.

% \medskip
\begin{definition}
  \label{dfn:relative cell complex}
  A pair of diffeological spaces $(X,A)$ is called a {smooth
    relative cell complex} if there is an ordinal $\delta$ and a
  $\delta$-sequence $Z \colon \delta \to \Diff$ such that the
  composition $Z_{0} \to \mbox{colim}Z$ coincides with the
  inclusion $i \colon A \to X$ and for each successor ordinal
  $\beta < \delta$, there is a smooth map
  $\phi_{\beta} \colon \bI^k \to Z_{\beta-1}$, called an
  {attaching map}, such that $Z_{\beta}$ is diffeomorphic to
  the adjunction space $Z_{\beta-1} \cup_{\phi_{\beta}} I^k$,
  % i.e.\ a pushout of % the diagram
  % \[
  %   Z_{\beta-1} \xleftarrow{\phi_{\beta}} \bI^k \hookrightarrow
  %   I^k
  % \]
  i.e.\ we have a pushout square
  \[
    \vcenter{%
      \xymatrix{%
        \bI^k \ar[d]_-{\cap} \ar[r]^-{\phi_{\beta}}
        & Z_{\beta-1} \ar[d]
        \\
        I^k \ar[r]^-{\Phi_{\beta}}
        & Z_{\beta},
      }%
    }%
  \]
  % that is, $Z_{\beta+1}$ is an adjunction space
  % $Z_{\beta} \cup_{(\phi_{\beta + 1},\,k_{d_{\beta + 1}})} I^{d_{\beta
  %     + 1}}$ for some $d_{\beta + 1} \geq 0$.
  In particular, if $A = \emptyset$ then $X$ is called a
  {smooth cell complex}.
\end{definition}
% \medskip

In $\Top$, relative cell complexes have a useful property
that they satisfy homotopy extension property.
Unfortunately, its smooth version no longer holds as can be
seen from the fact that $\bI \times I \cup I \times \{0\}$
is not a smooth deformation retract of $I^2$.  But we can
improve the situation by extending morphisms from smooth
maps to quasi-asymptotic ones.  More precisely, we have the
following.

% \medskip
\begin{theorem}
  \label{thm:HEP}
  Let $(X,A)$ be a smooth relative cell complex and
  $f \colon X \to Y$ be a quasi-asymptotic map.  Suppose
  there is a quasi-asymptotic homotopy
  $h \colon A \times I \to Y$ with $h_0 = f|A$.  Then $h$
  extends to a quasi-asymptotic homotopy
  $H \colon X \times I \to Y$ with $H_0 = f$ and
  $H|A \times I = h$.
\end{theorem}
% \medskip

To prove this, we need several lemmas.
For $k \geq 0$ let
\[
  L^k = \bI^k \times I \cup I^k \times \{0\} \subset
  I^{k+1}
\]
and $i \colon L^k \to I^{k+1}$ be the inclusion.

% \medskip
\begin{lemma}
  \label{lmm:asymptotic deformation retraction}
  $L^k$ is a deformation retract of $I^{k+1}$ in $\wDiff$.
\end{lemma}

\begin{proof}
  Let
  $g \colon (-1,2)^{k+1} \to \bI^k \times [0,2) \cup I^k
  \times \{0\}$ be the radial projection from
  $P = (1/2,\cdots,1/2,2)$.  Then $g$ is piecewise linear
  (hence piecewise smooth) with respect to the decomposition
  of $(-1,2)^{k+1}$ determined by the vertices of $I^{k+1}$
  and $P$.
  Hence $p = g|I^{k+1} \colon I^{k+1} \to L^k$ is a
  quasi-asymptotic map given as a restriction of
  $\widehat{p} = I_U(g)|\,\ras{I}^{k+1} \in
  \wCinf(\ras{I}^{k+1},\ras{\!L}^k)$.
  Moreover, for every $\sigma \in L^k(U)$ we have
  \[
    \widehat{p} \circ \ras{\sigma} = I_U(g) \circ
    \ras{\sigma} = \ras{(g \circ \sigma)} = \ras{\sigma} \in
    \ras{\!L}^k(\ras{U})
  \]
  implying that
  $\widehat{p} \colon \ras{I}^{k+1} \to \ras{\!L}^k$ is a
  deformation retraction with retracting homotopy
  $\ras{i} \circ \widehat{p} \simeq 1$ rel $\ras{\!L}^k$
  given by
  \[
    H(x,t) = (1-t)\widehat{p}(x) + tx, \quad (x,t) \in
    \ras{I}^{k+1} \times \ras{I}.
  \]
  Thus $p \colon I^{k+1} \to L^k$ is a deformation
  retraction in $\wDiff$ with retracting homotopy given by
  the quasi-asymptotic map $H|I^{k+1} \times I$.
\end{proof}

Clearly, the lemma above implies the following.

% \medskip
\begin{corollary}
  \label{crl:extendability_of_quasi_asymptotic_maps}
  Any quasi-asymptotic map $L^k \to X$ extends to a
  quasi-asymptotic map $I^{k+1} \to X$.
\end{corollary}
% \medskip

We also need the following lemma.

% \medskip
\begin{lemma}
  \label{lmm:subduction of pushout}
  Let $Z = Y \cup_\phi I^k$ be an adjunction space given by
  a smooth map $\phi \colon \bI^k \to Y$ and the inclusion
  $\bI^k \to I^k$.  Then the map
  \[
    \textstyle i \times 1 \bigcup \Phi \times 1 \colon Y \times I
    \coprod I^k \times I \to Z \times I
  \]
  induced by the natural maps $i \colon Y \to Z$ and
  $\Phi \colon I^k \to Z$ is a subduction.
\end{lemma}

\begin{proof}
  Let $P \colon U \to Z \times I$ be a plot of $Z \times I$ given
  by $P(r) = (\sigma(r),\sigma'(r))$, and let $r \in U$.  Since
  $\sigma$ is a plot of $Z$, there exists a plot
  $Q_1 \colon V \to Y$ such that $\sigma |V = i \circ Q_1$ holds,
  or a plot $Q_2 \colon V \to I^k$ such that
  $\sigma|V = \Phi \circ Q_2$ holds.  In either case, we have
  \[
    \textstyle P|V = (i \times 1 \bigcup \Phi \times 1) \circ
    (Q_{\alpha},\sigma')|V,
  \]
  where $\alpha$ is either $1$ or $2$.  This means
  $i \times 1 \bigcup \Phi \times 1$ is a subduction.
\end{proof}

We are now ready to prove Theorem~\ref{thm:HEP}.

% \begin{proof}[Proof of Theorem \ref{thm:HEP}]
Let $Z \colon \delta \to \Diff$ be a $\delta$-sequence such
that the composition $Z_0 \to \colim Z$ is the inclusion
$A \to X$.  We construct a quasi-asymptotic homotopy
$H \colon X \times I \to Y$ by transfinite induction on
$\beta < \delta$.
Suppose $\beta$ is a successor ordinal, and we already have
a quasi-asymptotic homotopy
$H_{\beta-1} \colon Z_{\beta-1} \times I \to Y$ satisfying
\[
  H_{\beta-1}|X \times \{0\} = f|Z_{\beta-1}, \quad
  H_{\beta-1}|A \times I = h
\]
Then we have a commutative diagram
\[
  \xymatrix@C=46pt@R=36pt{%
    \bI^k \times I \cup I^k \times \{0\}
    \ar[r]^-{(\phi_{\beta} \times 1) \cup \Phi_{\beta}}
    \ar[d] & Z_{\beta-1} \times I \cup \Phi_{\beta}(I^k)
    \times \{0\} \ar[r]^-{H_{\beta-1} \cup f} \ar[d] & Y
    \\
    I^k \times I \ar[r]^-{\Phi_{\beta} \times 1}
    \ar@{.>}[urr] & Z_{\beta} \times I \ar@{.>}[ur] }%
\]
By
Corollary~\ref{crl:extendability_of_quasi_asymptotic_maps}
the composition of upper arrows can be extended to
% a quasi-asymptotic map
$K_{\beta} \colon I^k \times I \to Y$, and we obtain a
commutative diagram
\[
  \xymatrix@C=50pt{%
    Z_{\beta-1} \times I \coprod I^k \times I
    \ar[r]^-{H_{\beta-1} \bigcup K_{\beta}}
    \ar[d]_-{i_{\beta} \times 1 \bigcup \Phi_{\beta} \times
      1} & Y
    \\
    Z_{\beta} \times I \ar[ru]_-{H_{\beta}} }%
\]
where $i_{\beta}$ is the inclusion
$Z_{\beta-1} \to Z_{\beta}$.  Since
$i_{\beta} \times 1 \bigcup \Phi_{\beta} \times 1$ is a
subduction by Lemma~\ref{lmm:subduction of pushout}, and
since $H_{\beta-1} \bigcup K_{\beta}$ is quasi-asymptotic,
we conclude by
Proposition~\ref{prp:subduction_and_quasi_asymptoticity}
that $H_{\beta}$ is quasi-asymptotic, too.

Thus, by transfinite induction on the $\delta$-sequence $Z$
we obtain a quasi-asymptotic homotopy
$H \colon X \times I \to Y$ extending
$h \colon A \times I \to Y$, completing the proof of
Theorem~\ref{thm:HEP}.
% \end{proof}

% \section*{Declarations}
% \textbf{Ethical Approval} \ Not applicable.
% \\ \\
% \textbf{Funding} \ No funding was received towards this
% work.  \\ \\
% \textbf{Availability of data and materials} \
% Not applicable.

%%%%%%%%%%%%%%%%%%%%%%%%%%%%%%%%%%%%%%%%%%%%%%%%%%%%%%%%%%%%
% References
%%%%%%%%%%%%%%%%%%%%%%%%%%%%%%%%%%%%%%%%%%%%%%%%%%%%%%%%%%%%
\bibliographystyle{plainnat}

\begin{thebibliography}{99}
\bibitem[Colombeau(2013)]{Colombeau} Colombeau, J.\ F.,
  2013, \emph{Nonlinear Generalized Functions: their origin,
    some developments and recent advances}, S\~{a}o Paulo
  J.\ Math.\ Sci., \textbf{7} (2013), 201--239.
\bibitem[Fr\"{o}hlicher and Kriegl(1988)]{Frohlicher-Kriegl}
  Fr\"{o}hlicher, A., and A.\ Kriegl, 1988, \emph{Linear
    spaces and differentiation theory}, Pure Appl.\ Math.,
  \textbf{13}, Wiley, Hoboken, NJ.
\bibitem[Giordano(2021)]{Giordano} Giordano, P., 2021,
  \emph{A Grothendieck topos of generalized functions I:
    Basic theory}, arXiv:2101.04492.
\bibitem[Giordano and Wu(2015)]{Giordano-Wu} Giordano, P.\
  and E.\ Wu, 2015, \emph{Categorical framework for
    generalized functions}, Arab.\ J.\ Math, \textbf{4}
  (2015), 301--328.
\bibitem[Grosser et al.(2001)]{GKOS} Grosser, M., M.\
  Kunzinger, M.\ Oberguggenberger and R.\ Steinbauer, 2001,
  \emph{Geometric theory of generalized functions with
    application to general relativity}, Math.\ Appl.,
  \textbf{537}, Springer-Science+Business Media, B.V.
\bibitem[Hovey(1999)]{Hovey} Hovey, M., 1999, \emph{Model
    categories}, Math.\ Surveys Monogr., \textbf{63},
  Amer.\ Math.\ Soc., Providence, RI.
\bibitem[Iglesias-Zemmour(2013)]{Zemmour} Iglesias-Zemmour,
  P., 2013, \emph{Diffeology}, Math.\ Surveys Monogr.,
  \textbf{165}, Amer.\ Math.\ Soc., Providence, RI.
\bibitem[Kock and Reyes(2006)]{Kock-Reyes} Kock, A.\ and G.\
  E.\ Reyes, 2006, \emph{Distributions and heat equation in
    SDG}, Cah.\ Topol.\ G\'{e}om.\ Diff\'{e}r.\ Cat\'{e}g.,
  \textbf{47} (2006), 2--28.
\bibitem[Kriegle and Michor(1997)]{Kriegle-Michor} Kriegl,
  A.\ and P.\ W.\ Michor, 1997, \emph{The convenient setting
    of global analysis}, Math.\ Surveys Monogr.,
  \textbf{53}, Amer.\ Math.\ Soc., Providence, RI.
\bibitem[Oberguggenberger and
  Todorov(1998)]{Oberguggenberger-Todorov} Oberguggenberger,
  M.\ and T.\ D.\ Todorov, 1998, \emph{An embedding of
    Schwartz distributions in the algebra of asymptotic
    functions}, Int.\ J.\ Math.\ Math.\ Sci., \textbf{21}
  (1998), 417--428.
\bibitem[Schwartz(1954)]{Schwartz} Schwartz, L., 1954,
  \emph{Sur l'impossibilit\'{e} de la multiplication des
    distributions}, C.\ R.\ Acad.\ Sci.\ Paris, \textbf{239}
  (1954), 847--848.
\end{thebibliography}
% \bibliographystyle{abbrvnnat}

%%%%%%%%%%%%%%%%%%%%%%%%%%%%%%%%%%%%%%%%%%%%%%%%%%%%%%%%%%%%%%%
\end{document}